\documentclass[a4paper]{article}
\usepackage[utf8]{inputenc}
\usepackage[english]{babel}
\usepackage{amsmath}
\usepackage{amsthm}
\usepackage{amssymb}
\usepackage{caption}
\usepackage{subcaption}
\usepackage{amsfonts}
\usepackage{cleveref}
\usepackage{mathtools}
\usepackage{graphicx}
\usepackage{xcolor}
\usepackage{latexsym}
\usepackage[bottom]{footmisc}
\usepackage{url}
\usepackage{algorithm}
\usepackage{algorithmic}
\usepackage{geometry}
\usepackage{tikz}

\newtheorem{theorem}{Theorem}[section]
\newtheorem{lemma}[theorem]{Lemma}
\newtheorem{remark}[theorem]{Remark}
\newtheorem{proposition}[theorem]{Proposition}

\DeclareSymbolFont{bbold}{U}{bbold}{m}{n}
\DeclareSymbolFontAlphabet{\mathbbold}{bbold}

\begin{document}

\title{A PAC algorithm in relative precision for bandit problem with costly sampling}

\author{M. Billaud-Friess\footnotemark[1] \and A. Macherey\footnotemark[1] \footnotemark[2]  \and A. Nouy\footnotemark[1] \and C. Prieur\footnotemark[2]}

\renewcommand{\thefootnote}{\fnsymbol{footnote}}
\footnotetext[1]{Nantes Université, Centrale  Nantes, LMJL, UMR CNRS 6629, 1 rue de la No\"e, 44321 Nantes}
\footnotetext[2]{Univ. Grenoble Alpes, Inria, CNRS, Grenoble INP*, LJK, 38000 Grenoble, France}
\renewcommand{\thefootnote}{\fnsymbol{footnote}}

\date{}

\maketitle

\begin{abstract}
This paper considers the problem of maximizing an expectation function over a finite set, or finite-arm bandit problem.  
We first propose a naive stochastic bandit algorithm for obtaining a probably approximately correct (PAC) solution to this discrete optimization problem in relative precision, that is a solution   
 which solves the optimization problem up to a relative error smaller than a prescribed tolerance, with high probability.
We also propose an adaptive stochastic bandit  algorithm which provides a PAC-solution with the same guarantees. The adaptive algorithm outperforms the mean complexity of the naive algorithm in terms of number of generated samples and is particularly well suited for applications with high sampling cost. \\


\textbf{Acknowledgements:} This research was partially supported by project MATHAMSUD FANTASTIC 20-MATH-05.

\end{abstract}			

\section{Introduction} 
\label{sec:introduction}

We consider an optimization problem 
\begin{equation}
\max_{\xi \in \Xi} ~ \mathbb{E}[Z(\xi)],
\label{eq:problem}
\end{equation}
where $\mathbb{E}[Z(\xi)]$ is the expectation of a random variable $Z(\xi)$, and where we assume that the set $\Xi$ is finite. Such a problem is encountered in different fields such as reinforcement learning \cite{Sutton2018} or robust optimization \cite{Beyer2007}. \\

\textcolor{black}{Problem \eqref{eq:problem} arises in clinical trials, for the identification of a treatment with the best effectiveness, see \cite{kuleshov2014algorithms}. In such a context $\mathbb{E}[Z(\xi)]$ corresponds to the expected effectiveness, with the response of one patient to the treatment $\xi$ which is a sample of $Z(\xi)$. A treatment $\xi$ gathers, for example, the proportions of the molecules used in the treatment. In that case the dimension of $\xi$ can be large, especially when several molecules have to be recombined to find the best treatment.} \\

To solve \eqref{eq:problem}, classical optimization methods include random search algorithms \cite{Gong2000,Yan1992}, stochastic approximation methods \cite{Dupa1982,Nemirovski2009} and bandit algorithms \cite{Lattimore2020,Audibert2010,Audibert2011,Garivier2011}. In this paper, we focus on unstructured stochastic bandit problems with a finite number of arms where "arms" stands for "random variables" and corresponds here to the $ Z(\xi)$, $\xi \in \Xi $ (see, e.g., \cite[Section 4]{Lattimore2020}). Stochastic means that the only way to learn about the probability distribution of arms $ Z(\xi)$, $\xi \in \Xi $ is to generate i.i.d. samples from it. Unstructured means that knowledge about the probability distribution of one arm $Z(\xi)$ does not restrict the range of possibilities for other arms $Z(\xi^{'})$, $\xi^{'}  \neq \xi$. \\

Additionally, we suppose here it is numerically costly to sample the random variables $Z(\xi)$, $\xi \in \Xi$. Our aim is thus to solve \eqref{eq:problem} by sampling as few as possible the random variables $Z(\xi)$, $\xi \in \Xi$. However, it is not feasible to solve \eqref{eq:problem} almost surely using only a finite number of samples from the random variables $Z(\xi)$, $\xi \in \Xi$. Thus, it is relevant to adopt a Probably Approximately Correct (PAC) approach (see e.g. \cite{Evendar2002}). For a precision $\tau_{abs}$ and a probability $\lambda \in (0,1)$, a $(\tau_{abs},\lambda)$-PAC algorithm returns $\hat{\xi}$ such that 
\begin{equation}
\mathbb{P} \left( \mathbb{E}[Z(\xi^{\star})] - \mathbb{E}[Z(\hat{\xi})] \le \tau_{abs} \right) \ge 1 - \lambda, \quad \xi^\star \in \arg \max_{\xi \in \Xi} ~ \mathbb{E}[Z(\xi)].
\label{eq:discreteabsolutePAC}
\end{equation}
Until recently, one of the main focus of bandit algorithms was the best arm (random variable) identification \cite{Audibert2010}, through the use of Successive Rejects algorithm or Upper Confidence Bounds algorithms. Such algorithms are $(0,\lambda)$-PAC algorithms, as stated in \cite{Evendar2002}. Racing algorithms \cite{Audibert2011} were designed to solve the best arm identification problem too and are mainly analyzed in a finite budget setting, which consists in fixing a maximum number of samples that can be used. 
While trying to identify the best arm, bandit algorithms also aim at  minimizing the regret \cite{audibert2009exploration,Bubeck2011,Garivier2011}. More recently, other focuses have emerged, such as the identification of the subset of $\Xi$ containing the $m$ best arms \cite{Kal2012,Kaufmann2016} or the identification of "good arms" (also known as thresholding bandit problem)  that are random variables whose expectation is greater or equal to a given threshold \cite{Kano2019,Tao2019,Muk2017,Locatelli2016}.  \\

The $(\tau_{abs},\lambda)$-PAC algorithms mentioned above measure the error in  absolute precision. However, without knowing $\mathbb{E}[Z(\xi^\star)]$, providing in advance a relevant value for $\tau_{abs}$ is not an easy task. In this work, we rather consider $(\tau,\lambda)$-PAC algorithms in relative precision that return $\hat \xi \in \Xi$ such that 
\begin{equation}
\mathbb{P} \left( \mathbb{E}[Z(\xi^{\star})] - \mathbb{E}[Z(\hat{\xi})] \le \tau | \mathbb{E}[Z(\xi^\star)]  | \right) \ge 1 - \lambda,
\label{eq:discreterelativePAC}
\end{equation}
where $ \tau$ and $\lambda$ are set in advance in $(0,1)$. 
We introduce two algorithms that yield a solution $\hat \xi$ satisfying \eqref{eq:discreterelativePAC}. The first algorithm builds an estimate precise enough for each expectation $\mathbb{E}[Z(\xi)]$. This naive approach drives a majority of the budget on the random variables with the lowest expectations in absolute value. In order to avoid this drawback and thus to reduce the number of samples required to reach the prescribed relative precision, we propose a second algorithm which adaptively samples random variables exploiting confidence intervals obtained from an empirical Berstein concentration inequality.  \\

The outline of the paper is as follows. In Section \ref{sec:1}, \textcolor{black}{we recall and extend to sub-Gaussian random variables} a Monte-Carlo estimate for the expectation of a single random variable that has been proposed in \cite{Mnihthesis}. It provides an estimation of the expectation with guaranteed relative precision, with high probability. In Section \ref{sec:2}, we introduce two new algorithms that rely on these Monte-Carlo estimates and yield 
a solution to \eqref{eq:discreterelativePAC}. Then, we study numerically the performance of our algorithms and compare them to algorithms from the literature, possibly adapted  to solve \eqref{eq:discreterelativePAC}.


\section{Monte-Carlo estimate with guaranteed relative precision}
\label{sec:1}
In what follows, we consider a random variable $Z$ defined on probability space $(\Omega,\mathcal{F},\mathbb{P})$.
We denote by $\overline{Z}_m$ the empirical mean of $Z$ and by $\overline{V}_m$ its empirical variance, respectively defined by
\begin{equation*}
\overline{Z}_m = \dfrac{1}{m} \sum_{i=1}^m Z_i \quad \text{and} \quad \overline{V}_m = \dfrac{1}{m} \sum_{i=1}^m \left( Z_i - \overline{Z}_m \right)^2,
\end{equation*}
where $(Z_i)_{i\ge 1}$ is a sequence  of i.i.d. copies of $Z$. The aim is to provide  an estimate $\hat{\mathbb{E}}[Z]$ of $\mathbb{E}[Z]$ which satisfies 
\begin{equation}
\mathbb{P} \left( | \hat{\mathbb{E}}[Z] - \mathbb{E}[Z] | \le \epsilon | \mathbb{E}[Z] | \right) \ge 1 - \delta,
\label{eq:epsilondeltaestimate1}
\end{equation}
with $(\epsilon,\delta) \in (0,1)^2$ given a priori. \textcolor{black}{We assume that the distribution of $Z$ satisfies the concentration inequality
\begin{align}
\mathbb{P} \left( | \overline{Z}_m - \mathbb{E}[Z] | \le c(m,x) \right) \ge 1 - x,
\label{eq:genericconcentrationinequality}
\end{align} 
where $x$ is in $(0,1)$ and $c(m,x)$ depends on $m$ and $x$. Below we recall standard concentration inequalities for sub-Gaussian\footnote{\textcolor{black}{A random variable $Z$ is said to be sub-Gaussian of parameter $\gamma > 0$ if, for all $s \ge 0$ we have
$\mathbb{E}[\exp(s(Z - \mathbb{E}[Z])] \le \exp(s^2 \gamma / 2).
$}} and bounded random variables. In the rest of the paper we denote by $SG(\gamma)$ the set of sub-Gaussian random variables of parameter $\gamma$ and by $B(a,b)$ the set of bounded random variables taking their values in the interval $[a,b]$.}
\textcolor{black}{\begin{theorem}
\label{th:hoeffding}
If $Z$ is in $SG(\gamma)$, \eqref{eq:genericconcentrationinequality} holds with $$c(m,x) = \sqrt{\frac{2 \gamma }{m} \log(2/x)}$$ for any $m \in \mathbb{N}$ and $x \in (0,1)$.
\end{theorem}}
\begin{theorem}\label{th:bernstein}
If $Z$ \textcolor{black}{is in $ B(a,b)$, \eqref{eq:genericconcentrationinequality} holds with $$c(m,x) = \sqrt{\frac{2\overline{V}_m \log(3/x)}{m}} + \frac{3(b-a) \log(3/x)}{m}$$} for any $m \in \mathbb{N}$ and $x \in (0,1)$. 
\end{theorem}
\begin{proof}
We simply apply \cite[Theorem 1]{audibert2009exploration} to $Z-a$ which is a positive random variable whose values are lower than $b-a$. 
\qed\end{proof}
Based on Theorem \ref{th:bernstein}, several estimates for $\mathbb{E}(Z)$ have been proposed in \cite{Mnihthesis,Mnih2008}. We focus in this paper on the estimate introduced in \cite[Equation (3.7)]{Mnihthesis} \textcolor{black}{and generalize it to any concentration inequality written as \eqref{eq:genericconcentrationinequality}.}

\subsection{Monte-Carlo estimate}
\label{ssec:21}
Considering a sequence $(d_m)_{m \ge 1}$ in $(0,1)$, we introduce the sequence $(c_m)_{m \ge 1}$ defined, for all $m \ge 1$, by 
\begin{equation}
\begin{split}
& \textcolor{black}{c_m =  c(m,d_m).} 
\end{split}
\label{eq:cm}
\end{equation}
\noindent Using \textcolor{black}{\eqref{eq:genericconcentrationinequality}}, we see that $c_m$ stands for the half-length of a confidence interval of level $1-d_m$ for $ \mathbb{E}[Z]$, i.e. 
\begin{equation}
\mathbb{P}(| \overline{Z}_m - \mathbb{E}[Z]| \le c_m) \ge 1 - d_m. \label{concentation-cm-dm}
\end{equation}
Let $M$ be an integer-valued random variable on $(\Omega,\mathcal{F},\mathbb{P})$, such that 
\begin{equation}
c_M \le \epsilon | \overline{Z}_M |,
\label{eq:Mnonmin}
\end{equation}
with $\epsilon \in (0,1)$.
Then, we define the following estimate
\begin{equation}
\begin{split}
\hat{\mathbb{E}}_{M}[Z] = \overline{Z}_{M} - \epsilon ~ \text{sign}{(\overline{Z}_{M})} c_{M}.
\end{split}
\label{estimateur:anyrandomvariable}
\end{equation}

\begin{proposition} 
Let $\epsilon, \delta \in (0,1)$. Assume that $Z$ takes its values in a bounded interval $[a,b]$ and 
that  $(d_m)_{m \ge 1} $ satisfies 
\begin{equation}
\sum_{m \ge 1} d_m \le \delta.
\label{assumption:sumdmledelta}
\end{equation}
Then the estimate $\hat{\mathbb{E}}_{M}[Z]$ defined by \eqref{estimateur:anyrandomvariable}, with $M$  satisfying \eqref{eq:Mnonmin}, is such that 
\begin{equation}
\mathbb{P} \left( \left| \hat{\mathbb{E}}_M[Z] - \mathbb{E}[Z] \right| \le \epsilon |\mathbb{E}[Z]| \right) \ge 1 - \delta.
\label{eq:epsilondeltaestimate}
\end{equation}
\label{prop:EBStopguarantee}
\end{proposition}

\begin{proof}
We have 
\begin{equation}
\begin{split}
\mathbb{P} \left( | \overline{Z}_{M} - \mathbb{E}[Z]| \le c_{M} \right) & \ge  \mathbb{P} \left( \cap_{m \ge 1} \left\{ | \overline{Z}_{m} - \mathbb{E}[Z]| \le c_{m} \right\} \right)  \\
&\ge 1 - \sum_{m=1}^{+ \infty} \mathbb{P} \left( | \overline{Z}_{m} - \mathbb{E}[Z]| > c_{m} \right) .
\end{split}
\end{equation}
Then using \eqref{concentation-cm-dm} and \eqref{assumption:sumdmledelta}, we deduce  that 
\begin{equation}
\begin{split}
\mathbb{P} \left( | \overline{Z}_{M} - \mathbb{E}[Z]| \le c_{M} \right) 
& \ge 1 - \sum_{m=1}^{+ \infty} d_m \ge 1 - \delta.
\end{split}
\label{eq:probatotaledelta}
\end{equation}
It remains to prove that $\{ | \overline{Z}_M - \mathbb{E}[Z] | \le c_M \}$ implies  $\{ | \hat{\mathbb{E}}_M[Z] - \mathbb{E}[Z] | \le \epsilon |  \mathbb{E}[Z] | \}$. 
In the rest of the proof, we assume that
$| \overline{Z}_{M} - \mathbb{E}[Z]| \le c_M$ holds. 
Let us recall that $c_M \le \epsilon | \overline{Z}_M |$.
Then,  since $\epsilon < 1$, we have 
$$| \overline{Z}_{M} - \mathbb{E}[Z]| \le c_M \le \epsilon | \overline{Z}_M | < | \overline{Z}_M |,$$ which implies that $\mathbb{E}[Z]$, $\overline{Z}_{M}$  and $\hat{\mathbb{E}}_M[Z]$ have the same sign. Therefore, 
\begin{equation*}
| \hat{\mathbb{E}}_{M}[Z] - \mathbb{E}[Z] | = \left | |\overline{Z}_{M} | - \epsilon c_{M} - | \mathbb{E}[Z] | \right|. 
\end{equation*}
It suffices to consider the case $\mathbb{E}[Z] > 0$ and we have
\begin{equation*}
\mathbb{E}[Z] \ge \overline{Z}_{M} - c_{M} \ge \dfrac{1-\epsilon}{\epsilon}c_{M}.
\end{equation*}
Therefore 
\begin{equation*}
\hat{\mathbb{E}}_{M}[Z] - \mathbb{E}[Z] = \overline{Z}_{M} - \mathbb{E}[Z] - \epsilon c_{M} \le (1 - \epsilon)c_{M} \le \epsilon \mathbb{E}[Z].
\end{equation*}
Also
\begin{equation*}
\mathbb{E}[Z] \le \overline{Z}_{M} + c_{M} \le (1 + \epsilon) \overline{Z}_{M},
\end{equation*}
and  
\begin{equation*}
\hat{\mathbb{E}}_{M}[Z] - \mathbb{E}[Z] \ge (1 - \epsilon^2)\overline{Z}_{M} - \mathbb{E}[Z] \ge (1 - \epsilon) \mathbb{E}[Z] - \mathbb{E}[Z] = - \epsilon \mathbb{E}[Z],
\end{equation*}
which concludes the proof. 
\qed\end{proof}
In practice, the computation of the estimate given by \eqref{estimateur:anyrandomvariable} requires a particular choice for the random variable $M$ and for the sequence $(d_m)_{m \ge 1}$. A natural choice for $M$ which satisfies \eqref{eq:Mnonmin} is 
\begin{equation}
M = \min \left\{ m \in \mathbb{N}^\star ~ : ~ c_m \le \epsilon | \overline{Z}_m | \right\}.
\label{eq:definitionM}
\end{equation}
\textcolor{black}{To ensure that $\hat{\mathbb{E}}_M[Z]$ can be built using a finite number of samples $Z_i$, we take a} sequence $(d_m)_{m \ge 1}$ such that 
\begin{equation}
\log(\textcolor{black}{1}/d_m)/m \underset{m \rightarrow + \infty}{{\longrightarrow}}0.\label{log3dm}
\end{equation} 
\textcolor{black}{Under this condition, with $c_m = c(m,d_m)$ defined in Theorem \ref{th:hoeffding} or Theorem \ref{th:bernstein}, we have that $c_m$ converges to $0$ almost surely. This condition together with $\mathbb{E}[Z] \neq 0$ are} sufficient to ensure that  $M < + \infty$ almost surely.

\begin{remark}
When choosing $M$ as in \eqref{eq:definitionM} \textcolor{black}{with $c_m = c(m,d_m)$ as defined in Theorem \ref{th:bernstein}}, the estimate defined by \eqref{estimateur:anyrandomvariable} is the one proposed in \cite[equation (3.7)]{Mnihthesis} for bounded random variables. A variant of this estimate can be found in \cite{Mnih2008}. 
\end{remark}

\begin{remark}
\textcolor{black}{Sometimes the sub-Gaussian parameter $\gamma$ (for sub-Gaussian random variables) or the exact value of $b-a$ (for bounded random variables) is not known. In this case, one can use any upper bound for $\gamma$ or $b-a$ and derive less tight concentration inequalities using Theorem \ref{th:hoeffding} and Theorem \ref{th:bernstein}.}
\end{remark}
\textcolor{black}{We propose to perform in the next section a complexity analysis for building an estimate in relative precision.}

\subsection{Complexity analysis}
\label{ssec:22}

In this section, we state \textcolor{black}{two complexity results, one for each concentration inequality we presented earlier (see Theorem \ref{th:hoeffding} and Theorem \ref{th:bernstein})}. Following \cite{Mnihthesis}, we focus on a particular sequence $(d_m)_{m \ge 1}$ defined by 
\begin{equation}
d_m = \delta c m^{-p} , \quad c = \frac{p-1}{p},\label{assumption:dmforme}
\end{equation}
which satisfies \eqref{assumption:sumdmledelta} and \eqref{log3dm}, for any $p>1$.
\textcolor{black}{\begin{proposition}[Complexity analysis for sub-Gaussian random variables]\label{prop:complexityMhoeffding}
Let $0< \delta \le 1$ and let $Z \in SG(\gamma)$ with expectation $\mu = \mathbb{E}[Z] $. If $\mu \neq 0$ and $(d_m)_{m \ge 1}$ satisfies \eqref{assumption:dmforme}, then $M$ defined by \eqref{eq:definitionM} with $c_m = c(m,d_m)$ as defined in Theorem \ref{th:hoeffding} satisfies $M < + \infty$ almost surely and 
\begin{equation}
\mathbb{P} \left( M >  \left \lceil \frac{2}{k} \log \left( \frac{2q}{k} \right) \right \rceil \right) \le \delta,
\label{eq:Mmajoreprobahoeffding}
\end{equation} 
where $ \lceil \cdot  \rceil$ denotes the ceil function, $k = \frac{\epsilon^2 \mu^2}{2 \gamma p (1 + \epsilon)^2}$ and $q = \sqrt[p]{\frac{2}{\delta c}}$. Moreover,
$$
\mathbb{E}(M) \le \left \lceil \frac{2}{k} \log \left( \frac{2q}{k} \right) \right \rceil  + \delta.
$$
\end{proposition}}
\begin{proof}
\textcolor{black}{Let us define the event $B  =  \bigcap_{m \ge 1} B_m$ with 
$$
B_m = \left\{ | \overline{Z}_m - \mu | \le c_m \right\}.
$$
A union bound argument and  Theorem \ref{th:hoeffding} with $x=\log(2/d_m)$ yield $\mathbb{P}(B) \ge 1 - \delta$. It remains to  prove that $B$ implies
\begin{equation}
 M \leq \left \lceil \frac{2}{k} \log \left( \frac{2q}{k} \right) \right \rceil \label{eq:majorationMhoeffding}
,\end{equation}
which will {prove} \eqref{eq:Mmajoreprobahoeffding}. In what follows, we suppose that $B$ holds.} \\
\textcolor{black}{We seek a bound for $M$ the smallest integer such that $c_M \le \epsilon | \overline{Z}_M |$. Defining $M^{\star}$ the smallest integer such that $c_M \le \frac{\epsilon | \mu |}{1 + \epsilon}$, we observe that, since $B$ holds, we have $c_{M^{\star}} \le \epsilon | \overline{Z}_{M^{\star}} |$ which gives $M \le M^{\star}$. Applying Lemma \ref{lemma:log} (stated in appendix) with $k = \frac{\epsilon^2 \mu^2}{2 \gamma p (1 + \epsilon)^2}$ and $q = \sqrt[p]{\frac{2}{\delta c}}$, we get
\begin{align*}
M^{\star} \le \frac{2}{k} \log(\frac{2q}{k}) \le \left \lceil \frac{2}{k} \log \left( \frac{2q}{k} \right) \right \rceil,
\end{align*}
which ends the proof of the first result.} \\
\textcolor{black}{Let us now prove the result in expectation. Let $K := \left \lceil \frac{2}{k} \log \left( \frac{2q}{k} \right) \right \rceil.$ We first note that 
$$
\mathbb{E}(M) = \sum_{j=0}^\infty \mathbb{P}(M>j) \le K + \sum_{j=K}^\infty \mathbb{P}(M>j) .
$$
If $M>j$, then $c_j >\epsilon \vert \bar Z_j \vert$. 
For $j \ge K$, we would like to prove that $c_j >\epsilon \vert \bar Z_j \vert$ implies $ (B_j)^c$, or equivalently that $B_j$ implies 
$c_j \le \epsilon \vert \bar Z_j \vert$. For $j\ge K$, $B_j$ implies $\vert \mu \vert \le \vert \bar Z_j \vert + c_j$ and $c_j \le \frac{\epsilon | \mu |}{1 + \epsilon}$ by Lemma \ref{lemma:log}. Combining the previous inequalities, we easily conclude that $B_j$ implies $c_j \le \epsilon \vert \bar Z_j \vert$.
For $j\ge K$, we then have 
$\mathbb{P}(M>j) \le \mathbb{P}(c_j >\epsilon \vert \bar Z_j \vert) \le   \mathbb{P}(B_j^c) \le d_j$, and then
$$
\mathbb{E}(M) \le K + \sum_{j=K}^\infty d_j \le K + \delta
,$$
which ends the proof.}
\qed\end{proof}

The following result extends the result of \cite[Theorem 2]{Mnihthesis} stated for random variables $Z$ with range in $[0,1]$. 

\begin{proposition}[\textcolor{black}{Complexity analysis for bounded random variables}]\label{prop:complexityMbernstein}
Let $0<\delta \le 3/4$ and let $Z \textcolor{black}{\in B(a,b)}$ with expectation $\mu = \mathbb{E}[Z] $ and variance $\sigma^2 = \mathbb{V}[Z]$. If $\mu\neq 0$ and $(d_m)_{m \ge 1}$ satisfies \eqref{assumption:dmforme}, then $M$ defined by \eqref{eq:definitionM} \textcolor{black}{with $c_m =c(m,d_m)$ as defined  in Theorem \ref{th:bernstein}} satisfies $M < + \infty$ almost surely and 
\begin{equation}
\mathbb{P} \left( M>  \left \lceil \dfrac{2}{\nu} \left( p \log \left( \dfrac{2p}{\nu} \right) + \log \left( \dfrac{3}{c \delta} \right)  \right)  \right \rceil \right) \le 4 \delta/3,
\label{eq:Mmajoreproba}
\end{equation} 
where $ \lceil \cdot  \rceil$ denotes the ceil function and $$\nu = \min \left (   \frac{\max(\sigma^2,\epsilon^2 \mu^2)}{(b-a)^2}  , \frac{\epsilon^2 \mu^2}{(1+\epsilon)^2 \max(\sigma^2,\epsilon^2 \mu^2) \gamma}  \right ),$$
with $\gamma=(\sqrt{2 + 2\sqrt{2} + 2/3}+3)^2$.
Moreover,
$$
\mathbb{E}(M) \le \left \lceil \dfrac{2}{\nu} \left( p \log \left( \dfrac{2p}{\nu} \right) + \log \left( \dfrac{3}{c \delta} \right)  \right)  \right \rceil  + 4\delta/3.
$$
\end{proposition}

\begin{proof}
See Appendix. 
\qed\end{proof}

\begin{remark}\label{rem:complexity}
The result from Proposition \ref{prop:complexityMbernstein} helps in understanding the influence of parameters $(\varepsilon,\delta)$ appearing in \eqref{eq:epsilondeltaestimate1} on $M$. Indeed,  we deduce from this result that for $\delta<1/2$, 
$$
\mathbb{E}(M) \lesssim \nu^{-1} \log(\nu^{-1}) + (\nu^{-1}+1)\log(\delta^{-1}).
$$ 
We first observe a weak impact of $\delta$ on the average complexity. When $\epsilon \to 0$, we have  $\nu \sim \epsilon^{2}\frac{\mu^2}{ \sigma^2 \gamma}$. Then for fixed $\delta$ and $\epsilon\to 0$, the bound for $\mathbb{E}(M)$ is in $O(\epsilon^{-2}\frac{ \sigma^2}{\mu^2})$. As expected, the relative precision $\epsilon$ has a much stronger impact on the average complexity. 
\end{remark}

\section{Optimization algorithms with guaranteed relative precision}
\label{sec:2}

In this section we consider a finite collection of random variables $Z(\xi)$ on $(\Omega,\mathcal{F},\mathbb{P})$, indexed by $\xi \in \Xi$, and such that $\mathbb{E}[Z(\xi)] \neq 0$. We denote by $\overline{Z(\xi)}_m$ the empirical mean of $Z(\xi)$ and $\overline{V(\xi)}_m$ its empirical variance, respectively defined by
\begin{equation}
\overline{Z(\xi)}_m  = \frac{1}{m} \sum_{i=1}^{m} Z(\xi)_i \quad \text{and} \quad \overline{V(\xi)}_{m}  = \frac{1}{{m}} \sum_{i=1}^{m} \left( Z(\xi)_i - \overline{Z(\xi)}_{m} \right)^2,
\end{equation} 
where $\left\{ (Z(\xi)_i)_{i\ge 1} : \xi \in \Xi \right\}$ are independent i.i.d. copies of $\left\{ Z(\xi) : \xi \in \Xi \right\}$. We assume that 
\begin{align*}
\mathbb{P} \left( | \overline{Z(\xi)}_m - \mathbb{E}[Z(\xi)] | \le c(m,x,\xi) \right) \ge 1 - x,
\end{align*}
for each $\xi \in \Xi$, $0 \le x \le 1$ and $m \ge 1$. \textcolor{black}{We introduce then $\# \Xi$ different sequences 
$$
c_{\xi,m} = c(m,d_m,\xi),
$$
where $(d_m)_{m \ge 1}$ is a positive sequence, independent from $\xi$, such that $\sum_{m \ge 1} d_m \le \delta$. For example  we take $c_{\xi,m} =\sqrt{\frac{2 \overline{V(\xi)}_{m} \log(3/d_{m})}{{m}}} + \frac{3 \left( b(\xi) - a(\xi) \right) \log(3/d_{m})}{{m}}$ for bounded random variables and $c_{\xi,m} = \sqrt{\frac{2 \gamma(\xi) }{m} \log(2/d_m)}$ for sub-Gaussian random variables.} \\
Taking $\epsilon$ in $(0,1)$, for each $\xi$ in $\Xi$, we define, as in \eqref{eq:definitionM}, 
\begin{equation}
m(\xi) = \min \left\{ m \in \mathbb{N}^\star ~ : ~ c_{\xi,m} \le \epsilon | \overline{Z(\xi)}_m | \right\}.
\label{eq:definitionMxi}
\end{equation}
Then defining $s(\xi) := \text{sign} ( \overline{Z(\xi)}_{m(\xi)} )$, we propose the following estimate for $\mathbb{E}[Z(\xi)]$:
\begin{equation}
\hat{\mathbb{E}}_{{m(\xi)}}[Z(\xi)] = \overline{Z(\xi)}_{{M(\xi)}} - \epsilon ~ s(\xi) c_{\xi,{m(\xi)}}.
\label{eq:estimateurZxi}
\end{equation}
\newline
These notation being introduced, we propose below two algorithms returning $\hat{\xi}$ in $\Xi$ such that
\begin{equation}
\mathbb{P} \left( \mathbb{E}[Z(\xi^\star)] - \mathbb{E}[Z(\hat{\xi})] \le \tau | \mathbb{E}[Z(\xi^\star)] | \right) \ge 1 - \lambda, \quad \xi^\star \in \arg \max_{\xi \in \Xi}\mathbb{E}[Z(\xi)],
\label{eq:PACproblem}
\end{equation}
for given $(\tau,\lambda)$ in $(0,1)^2$. 

\subsection{Non-adaptive algorithm}

We first propose a non-adaptive algorithm that provides a parameter $\hat \xi$ satisfying \eqref{eq:PACproblem}, by selecting the maximizer of independent estimates $\hat{\mathbb{E}}_{m(\xi)}[Z(\xi)]$  of  $\mathbb{E}[Z(\xi)]$ over $\Xi$.
\begin{algorithm}[h]\normalsize
\caption{Non-adaptive \textcolor{black}{algorithm}}
\label{alg:nonadaptatif}
\begin{algorithmic}[1]
\REQUIRE $\tau$, $\lambda$, $\{Z(\xi)\}_{\xi \in \Xi}$.
\ENSURE $\hat \xi$ 
\STATE Set $\epsilon = \frac{\tau}{2 + \tau}$ and $\delta = \lambda/ \# \Xi$. 
\FORALL{$\xi \in \Xi$}
\STATE Build an estimate $\hat{\mathbb{E}}_{m(\xi)}[Z(\xi)]$ of $\mathbb{E}[Z(\xi)]$ using \eqref{eq:estimateurZxi} with $\epsilon$ and $\delta$ as above.
\ENDFOR
\STATE Select $\hat \xi$ such that
$$
\hat{\xi} \in \underset{\xi \in \Xi}{\arg \max} ~ \hat{\mathbb{E}}_{m(\xi)} \left[ Z(\xi) \right]. 
$$
\end{algorithmic}
\end{algorithm} 

\begin{proposition} Let $(\tau,\lambda) \in (0,1)^2$. We assume that, for all $\xi \in \Xi$,  $Z(\xi)$ is a random variable with $\mathbb{E}[Z(\xi)] \neq 0$. Moreover we assume that the sequence $(d_m)_{m \ge 1}$ is such that
\begin{equation}
\sum_{m=1}^{+ \infty} d_m \le  \frac{\lambda}{\# \Xi} := \delta  \quad \text{and} \quad \log(1/d_m)/m \underset{m \rightarrow + \infty}{\rightarrow} 0.
\label{eq:conditionsurdmcomplete1}
\end{equation}
Then, for all $\xi$ in $\Xi$, the estimate $\hat{\mathbb{E}}_{m(\xi)}[Z(\xi)]$ is well defined and satisfies
\begin{equation}
\mathbb{P} \left( \left| \mathbb{E}[Z(\xi)] - \hat{\mathbb{E}}_{m(\xi)}[Z(\xi)] \right| \le \epsilon \left| \mathbb{E}[Z(\xi)] \right| \right) \ge 1 - \delta,
\label{eq:guarantiesestimateur}
\end{equation}
with $\epsilon = \frac{\tau}{2 + \tau}$. Moreover, the value $\hat{\xi}$ returned by Algorithm \ref{alg:nonadaptatif} satisfies \eqref{eq:PACproblem}. 
\label{prop:naiveguarantee} 
\end{proposition}
\begin{proof}
The assumptions on $(d_m)_{m \ge 1}$ in \eqref{eq:conditionsurdmcomplete1} combined with $\mathbb{E}[Z(\xi)] \neq 0$ ensure that for all $\xi$ in $\Xi$ , $M(\xi)$ is almost surely finite. Then, for all $\xi$ in $\Xi$, the estimate $\hat{\mathbb{E}}_{m(\xi)}[Z(\xi)]$ is well defined. Applying  Proposition \ref{prop:EBStopguarantee} for each $Z(\xi)$ with $\delta = \lambda / \# \Xi$ and $\epsilon = \frac{\tau}{2 + \tau}$, we obtain \eqref{eq:guarantiesestimateur}.\\
Now let $A(\xi) = \left\{ \left| \mathbb{E} \left[ Z(\xi) \right] - \hat{\mathbb{E}}_{m(\xi)} \left[ Z(\xi) \right] \right| \le  \epsilon \left| \mathbb{E} \left[ Z(\xi) \right] \right| \right\}$. By \eqref{eq:guarantiesestimateur}, $\mathbb{P}(A(\xi)) \ge 1 - \frac{\lambda}{\#\Xi}$ and by a union bound argument, $\mathbb{P}(\cap_{\xi \in \Xi} A(\xi)) \ge 1 - \lambda$. To prove that $\hat{\xi}$ satisfies \eqref{eq:PACproblem}, 
it remains to prove that $\cap_{\xi \in \Xi} A(\xi)$ implies $\mathbb{E} \left[ Z(\xi^\star) \right] - \mathbb{E} [Z(\hat \xi)] \le \tau \left| \mathbb{E} \left[ Z(\xi^\star) \right] \right|$. In what follows we suppose that $\cap_{\xi \in \Xi} A(\xi)$ holds.  Since $\epsilon < 1$, 
 $\mathbb{E}[Z(\xi)]$,  $\hat{\mathbb{E}}_{m(\xi)}[Z(\xi)]$ and $\overline{Z(\xi)}_{m(\xi)}$ have the same sign, that we denote by $s(\xi)$. Since $A(\xi^\star) \cap A(\hat{\xi})$ holds, we have 
\begin{equation*}
\begin{split}
\mathbb{E} \left[ Z(\xi^\star) \right] - \mathbb{E} [Z(\hat \xi)] &\le \mathbb{E} \left[ Z(\xi^\star) \right] - \dfrac{\hat{\mathbb{E}}_{m(\hat{\xi})} [Z(\hat \xi)]}{1 + s(\hat{\xi})\epsilon} \le \mathbb{E} \left[ Z(\xi^\star) \right] - \dfrac{\hat{\mathbb{E}}_{m(\xi^\star)} \left[ Z( \xi^\star) \right]}{1 + s(\hat{\xi})\epsilon} \\
& \le \mathbb{E} \left[ Z(\xi^\star) \right] - \dfrac{1 - s(\xi^\star)\epsilon}{1 + s(\hat{\xi})\epsilon} \mathbb{E} \left[ Z(\xi^\star) \right] 
 \\
 &= \frac{\epsilon (s(\xi^\star) + s(\hat{\xi}))}{1 + s(\hat{\xi})\epsilon} \mathbb{E} \left[ Z(\xi^\star) \right]. 
\end{split}
\end{equation*}
Then we deduce 
\begin{equation}
\mathbb{E} \left[ Z(\xi^\star) \right] - \mathbb{E} [Z(\hat \xi)]  \le \frac{2 \epsilon}{1 - \epsilon} | \mathbb{E}[Z(\xi^\star)] | = \tau | \mathbb{E}[Z(\xi^\star)] |,
\label{final-guarantee-naive}
\end{equation}
which ends the proof. \qed
\end{proof}
\begin{remark}\label{rem:alg-nonadapt-positive}
If $\mathbb{E}[Z(\xi^\star)] > 0$, we can prove that $s(\hat{\xi}) = s(\xi^\star) =1$, so that the inequality \eqref{final-guarantee-naive} becomes 
\begin{equation*}
\mathbb{E} \left[ Z(\xi^\star) \right] - \mathbb{E} [Z(\hat \xi)]  \le \frac{2\epsilon}{1 + \epsilon} | \mathbb{E}[Z(\xi^\star)] | .
\end{equation*}
Therefore, we can set $\epsilon = \frac{\tau}{2 - \tau}$ in Algorithm \ref{alg:nonadaptatif} to lower the complexity and still guarantee 
that $\hat{\xi}$ satisfies \eqref{eq:PACproblem}.
\end{remark}
Algorithm \ref{alg:nonadaptatif} provides  for each random variable an estimate $\hat{\mathbb{E}}_{m( {\xi})} [Z(  \xi)]$ that satisfies \eqref{eq:guarantiesestimateur}.  However, as will be illustrated later, this algorithm tends to use many samples for variables with a low expectation in absolute value.  We propose in the next subsection an adaptive algorithm avoiding this drawback by using confidence intervals, which results in a lower overall complexity. 

\subsection{Adaptive algorithm}

The idea of the adaptive algorithm is to successively increase the number of samples $m(\xi)$ of a subset of random variables $Z(\xi)$ that are selected based on confidence intervals of $\mathbb{E}[Z(\xi)]$ deduced from  the concentration inequality of Theorem \ref{th:hoeffding} \textcolor{black}{or Theorem \ref{th:bernstein}}. This algorithm follows the main lines of  the racing  algorithms \cite[Section 4]{Mnih2008}. However  racing algorithms do not allow to sample again a random variable discarded in an earlier step of the algorithm. The adaptive algorithm presented hereafter allow it. \\ 

In order to present this adaptive algorithm, for each $\xi$, we introduce the confidence interval $[ \beta^-_{m(\xi)}(\xi), \beta^+_{m(\xi)}(\xi)]$, with  
\begin{equation}
\beta^{-}_{\xi,m(\xi)} = \overline{Z(\xi)}_{m(\xi)} - c_{\xi,{m(\xi)}} \quad \text{and} \quad  \beta^{+}_{\xi,m(\xi)} = \overline{Z(\xi)}_{m(\xi)} + c_{\xi,{m(\xi)}}. 
\label{eq:beta+-def1}
\end{equation}
From \eqref{concentation-cm-dm}, we have that 
\begin{equation}
\mathbb{P}\left( \beta^{-}_{\xi,m(\xi)} \le  \mathbb{E}(Z(\xi)) \le \beta^{+}_{\xi,m(\xi)}  \right) \ge 1- d_{m(\xi)}.
\label{eq:betaconfint}
\end{equation}
We define $\epsilon_{\xi,m(\xi)}$ by 
\begin{equation*}
\epsilon_{\xi,m(\xi)} = \frac{c_{\xi,{m(\xi)}}}{|\overline{Z(\xi)}_{m(\xi)}|} 
\end{equation*}
if $\overline{Z(\xi)}_{m(\xi)} \neq 0$, or $\epsilon_{\xi,m(\xi)} = +\infty$ otherwise. 
Letting $s(\xi) := \text{sign} ( \overline{Z(\xi)}_{m(\xi)} )$, we use as an estimate for $\mathbb{E}[Z(\xi)]$
\begin{equation}
\hat{\mathbb{E}}_{{m(\xi)}}[Z(\xi)] = \left\{
    \begin{array}{ll}
        \overline{Z(\xi)}_{{m(\xi)}} - \epsilon_{\xi,m(\xi)} ~ s
        (\xi) c_{\xi,{m(\xi)}} & \mbox{if } \epsilon_{\xi,m(\xi)} < 1, \\
        \overline{Z(\xi)}_{m(\xi)} & \mbox{otherwise.}
    \end{array}
\right.
\label{eq:newestimate}
\end{equation}
If $\epsilon_{\xi,m(\xi)} < 1$, we note that 
\begin{equation*}
\hat{\mathbb{E}}_{{m(\xi)}}[Z(\xi)] = ( \overline{Z(\xi)}_{m(\xi)} \mp c_{\xi,m(\xi)} ) \left( 1 \pm s(\xi)\epsilon_{\xi,m(\xi)} \right),
\end{equation*}
so that 
\begin{equation}
\frac{\hat{\mathbb{E}}_{{m(\xi)}}[Z(\xi)]}{1 \pm s(\xi)\epsilon_{\xi,m(\xi)}} = \beta^{\mp}_{\xi,m(\xi)}. 
\label{eq:beta+-def2}
\end{equation}
The adaptive algorithm is described in Algorithm \ref{alg:bandit}. At each iteration $n$, one sample of $Z(\xi)$ is drawn 
for each $\xi$ in a subset $\Xi_n$ selected according to \eqref{eq:enrichingcondition}. 
\begin{algorithm}[h]
\caption{Adaptive \textcolor{black}{algorithm}}
\label{alg:bandit}\normalsize
\begin{algorithmic}[1]
\REQUIRE $\tau$, $\lambda$, $\{Z(\xi)\}_{\xi \in \Xi}$.
\ENSURE $\hat \xi$.
\STATE Set $n=0$, $\Xi_0 = \Xi$, $\epsilon_{\xi,0} = +\infty$  and $m(\xi) = 0$ for all $\xi \in \Xi$. 
\WHILE{$\# \Xi_n > 1$ \textbf{and} $\underset{\xi \in \Xi_n}{\max} ~ \epsilon_{\xi,m(\xi)} > \dfrac{\tau}{2+\tau}$}
\FORALL{$\xi \in \Xi_n$} 
\STATE Sample $Z(\xi)$, increment $m(\xi)$ and update $\epsilon_{\xi,m(\xi)}$. 
\STATE Build the estimate $\hat{\mathbb{E}}_{m(\xi)}[Z(\xi)]$ using \eqref{eq:newestimate}.
\ENDFOR
\STATE Increment $n$ and put in $\Xi_n$ every $\xi \in \Xi$ such that
\begin{equation}
\beta^{+}_{\xi,m(\xi)} \ge \underset{\nu \in \Xi}{\max} ~ \beta^{-}_{\nu,m(\nu)}(\nu).
\label{eq:enrichingcondition}
\end{equation} 
\ENDWHILE
\STATE Select $\hat \xi$ such that 
$$
\hat{\xi} \in \underset{\xi \in \Xi_n}{\arg \max} ~ \hat{\mathbb{E}}_{m(\xi)} \left[ Z(\xi) \right]. 
$$
\end{algorithmic}
\end{algorithm} 
In the next proposition, we prove that the algorithm returns a solution to \eqref{eq:PACproblem} under suitable assumptions.
\begin{proposition}
Let $(\tau,\lambda) \in (0,1)^2$. We assume that $(d_m)_{m \ge 1}$ is a positive sequence that satisfies
\begin{equation}
\sum_{m=1}^{+ \infty} d_m \le \frac{\lambda}{\# \Xi} \quad \text{and} \quad \log(1/d_m)/m \underset{m \rightarrow + \infty}{\rightarrow} 0.
\label{eq:conditionsurdmcomplete}
\end{equation}
Moreover, we assume that, for all $\xi$ in $\Xi$, $Z(\xi)$ is a random variable with $\mathbb{E}[Z(\xi)] \neq 0$. Then, it holds almost surely that Algorithm \ref{alg:bandit} stops and $\hat \xi$ satisfies \eqref{eq:PACproblem}. 
\label{prop:banditresult}
\end{proposition}
\begin{proof}
Let $m_n(\xi)$ denote the number of samples of $Z(\xi)$ at iteration $n$ of the algorithm. 
We first prove by contradiction that Algorithm \ref{alg:bandit} stops almost surely. Let us suppose that Algorithm \ref{alg:bandit} does not stop with probability $\eta>0$, that means 
\begin{equation}
\mathbb{P}\left(\forall n>0, ~ \# \Xi_n > 1 \quad \text{and} \quad \max_{\xi \in \Xi_n} \epsilon_{\xi,m_n(\xi)} > \frac{\tau}{2 + \tau}\right) = \eta>0.
\label{eq:nonstoppingcondition} 
\end{equation} 
Since $(\Xi_n)_{n \ge 1}$ is a sequence from  a finite set, we can extract a constant subsequence, still denoted $(\Xi_n)_{n \ge 1}$, equal to $\Xi^\star \subset \Xi$, with $\Xi^\star \neq \emptyset$ such that 
\begin{equation}
\mathbb{P}\left( \forall n>1,\; \max_{\xi \in \Xi^\star} \epsilon_{\xi,m_n(\xi)} > \frac{\tau}{2 + \tau}\right) \ge \eta >0  .
\label{eq:nonstoppingcondition-Xistar} 
\end{equation} 
Since $\log(3/d_m)/m \to 0$ as $m\to +\infty$ and $\mathbb{E}[Z(\xi)] \neq 0$ for all $\xi$, we have that $\epsilon_{\xi,m} \overset{a.s}{\underset{m \rightarrow + \infty}{\rightarrow}} 0$ for all $\xi$ in $\Xi$. Yet, since at each iteration $n$ (from the subsequence), we increase $m_n(\xi)$ for all $\xi$ in {$\Xi^\star$}, we have that $m_n(\xi) \to \infty$ as $n\to \infty$ for all $\xi\in \Xi^\star$.
Therefore, 
$ \lim_{n\to+\infty} \max_{\xi \in \Xi^{\star}} \epsilon_{\xi,m_n(\xi)} = 0$ holds almost surely, which contradicts \eqref{eq:nonstoppingcondition-Xistar}.
\\

 We now prove that $\hat{\xi}$ satisfies \eqref{eq:PACproblem}. For clarity, we remove the index $n$ from $m_n(\xi).$ 
 Defining $A(\xi) = \left  \{ | \overline{Z(\xi)}_{m(\xi)} - \mathbb{E}[Z(\xi)]  | \le c_{\xi,m(\xi)} \right \}$ for all $\xi$ in $\Xi$, we proceed as in \eqref{eq:probatotaledelta} to obtain
\begin{equation*}
\mathbb{P} \left(A(\xi) \right) \ge 1 - \lambda/ \# \Xi. 
\end{equation*}
Thus, by a union bound argument
\begin{equation*}
\mathbb{P} \left( \cap_{\xi \in \Xi} A(\xi)  \right) \ge 1 - \lambda. 
\end{equation*}
It remains to prove that $\cap_{\xi \in \Xi}  A(\xi)$ implies $\mathbb{E}[Z(\xi^\star)] - \mathbb{E}[Z(\hat{\xi})] \le \tau  |\mathbb{E}[Z(\xi^\star)] |$ in order to prove  that $\hat \xi$ satisfies \eqref{eq:PACproblem}. In the rest of the proof, we suppose that $\cap_{\xi \in \Xi} A(\xi)$ holds. 
First, for all $\xi \notin \Xi_n$, using \eqref{eq:beta+-def1}, we have
\begin{equation}
\mathbb{E}[Z(\xi)] \le \beta^+_{\xi,m(\xi)} < \max_{\nu \in \Xi}\beta^-_{\nu,m(\nu)} \le \max_{\nu \in \Xi} \mathbb{E}[Z(\nu)] = \mathbb{E}[Z(\xi^\star)],
\end{equation}
that implies $\xi^\star \in \Xi_n$. If the stopping condition is $\# \Xi_n = 1$, we then have $\hat{\xi} = \xi^\star$. 
If the stopping condition is $\max_{\xi \in \Xi_n} ~ \epsilon_{\xi,m(\xi)} \le \frac{\tau}{2+\tau} < 1$, it means that, for all $\xi$ in $\Xi_n$, $\epsilon_{\xi,m(\xi)} \le \frac{\tau}{2+\tau} < 1$. Then for  all $\xi \in \Xi_n$, using Proposition \ref{prop:EBStopguarantee} with $\epsilon = \epsilon_{\xi,m(\xi)} < 1$ and $\delta = \lambda / \# \Xi < 1$ and the fact that $\cap_{\xi \in \Xi} A(\xi)$ holds, we obtain that    the  estimate $\hat{\mathbb{E}}_{{m(\xi)}}[Z(\xi)]$ satisfies 
\begin{equation}
\left| \hat{\mathbb{E}}_{{m(\xi)}}[Z(\xi)] - \mathbb{E}[Z(\xi)] \right| \le \epsilon_{\xi,m(\xi)} | \mathbb{E}[Z(\xi)] |.
\label{eq:relativeprecision}
\end{equation}
We have that $\epsilon_{\xi,m(\xi)}<1$ and \eqref{eq:beta+-def2} hold for all $\xi  \in \Xi_n$. In particular, since $\hat \xi, \xi^\star \in \Xi_n$  we get 
\begin{equation*}
\begin{split}
\mathbb{E} \left[ Z(\xi^\star) \right] - \mathbb{E} [Z(\hat \xi)] & \le \mathbb{E} \left[ Z(\xi^\star) \right] - \dfrac{\hat{\mathbb{E}}_{m(\hat \xi)} [Z(\hat \xi)]}{1 + s(\hat{\xi})\epsilon_{\hat{\xi},m(\hat{\xi})}} 
 \le \mathbb{E} \left[ Z(\xi^\star) \right] - \dfrac{\hat{\mathbb{E}}_{m(\xi^\star)} \left[ Z( \xi^\star) \right]}{1 + s(\hat{\xi})\epsilon_{\hat{\xi},m(\hat{\xi})}} \\
& \le \mathbb{E} \left[ Z(\xi^\star) \right] - \dfrac{1 - s(\xi^\star)\epsilon_{\xi^\star,m(\xi^\star)}}{1 + s(\hat{\xi})\epsilon_{\hat{\xi},m(\hat{\xi})}} \mathbb{E} \left[ Z(\xi^\star) \right] \\
& = \frac{s(\xi^\star)\epsilon_{\xi^\star,m(\xi^\star)} + s(\hat{\xi})\epsilon_{\hat{\xi},m(\hat{\xi})}}{1 + s(\hat{\xi})\epsilon_{\hat{\xi},m(\hat{\xi})}} \mathbb{E} \left[ Z(\xi^\star) \right].
\end{split}
\end{equation*}
Then we deduce
\begin{equation}
\mathbb{E} \left[ Z(\xi^\star) \right] - \mathbb{E} [Z(\hat \xi)]  \le \frac{2 \tau/(2 + \tau)}{1 - \tau/(2 + \tau)} | \mathbb{E}[Z(\xi^\star)] | = \tau | \mathbb{E}[Z(\xi^\star)] | ,\label{final-guarantee}
\end{equation}
which ends the proof.\qed
\end{proof}
\begin{remark}
As for Algorithm \ref{alg:nonadaptatif} (see Remark \ref{rem:alg-nonadapt-positive}), if $\mathbb{E}[Z(\xi^\star)] > 0$, we can set $\epsilon= \frac{\tau}{2 - \tau}$ in Algorithm \ref{alg:bandit} to lower the complexity and still guarantee that $\hat \xi$ satisfies \eqref{eq:PACproblem}. 
\end{remark}
\begin{remark}
A variant of Algorithm \ref{alg:bandit} using batch sampling would generate several samples of $Z(\xi)$ at step 4.  
The result of Proposition \ref{prop:banditresult} also holds for the algorithm with batch sampling. An optimal choice of the number of samples should depend on sampling costs. 
\end{remark}
\textcolor{black}{A theoretical upper bound for the complexity of Algorithm \ref{alg:nonadaptatif} directly follows from Proposition \ref{prop:complexityMhoeffding} or Proposition \ref{prop:complexityMbernstein}. The analysis of the adaptive Algorithm \ref{alg:bandit} is not straightforward. The complexity of algorithms involving similar adaptive sampling strategies in absolute precision has been studied in several papers for different stopping rules: \cite{kaufmann2013information} provides an upper bound for the complexity that holds in probability while \cite{Kaufmann2016} provides a lower bound for this complexity. In these papers the authors were searching for $m$ variables in the set $\mathcal{S}_{m,\tau}$ of $(\tau,m)$-optimal random variables that are random variables whose mean $\mathbb{E}[Z(\xi)]$ is greater than $\mathbb{E}[Z(\xi_m)] - \tau$, where $\mathbb{E}[Z(\xi_m)]$ is the $m^{th}$ highest mean. In the case $m = 1$, finding one element of $\mathcal{S}_{1,\tau}$ with a given probability $1 - \lambda$, is equivalent to find $\hat{\xi}$ that satisfies \eqref{eq:discreteabsolutePAC} with $\tau_{abs} = \tau$. The critical point for a theoretical complexity  analysis of our adaptive algorithm is to adapt these analyses to relative precision. This could be the concern of a future work. } \\



\section{Numerical results}\label{sec:results}
In this section, we propose a numerical study of the behaviour of our algorithms on \textcolor{black}{two} toy examples.  \\

\subsection{First test case}
\textcolor{black}{We first apply our algorithms to the framework of bounded random variables. We use the concentration inequality bounds from Theorem \ref{th:bernstein}. W}e consider the set of random variables 
$Z(\xi) = f(\xi) + U(\xi), \; \xi \in \Xi$, 
where $f(\xi) = \sin(\xi) + \sin(10\xi/3)$, 
the $U(\xi)$ are i.i.d. uniform random variables over $(-1/20,1/20)$, and $\Xi = \{ 3 + 4i/100 : 0\le i\le 100\}$. The numerical results are obtained with the sequence $(d_m)_{m \ge 1}$ defined by \eqref{assumption:dmforme} with $p=2$. We set $\tau =0.1 $ and $ \lambda = 0.1$. \\

We first compare our algorithms with two existing ones. The first one is the Median Elimination (ME) algorithm (see \cite{Evendar2002} for a description of the algorithm), that solves problem \eqref{eq:discreteabsolutePAC}. We take $\tau_{abs} = \tau | \mathbb{E}[Z(\xi^\star)] |$ to ensure ME  algorithm provides a solution that also guarantees \eqref{eq:discreterelativePAC}. Of course, this is not feasible in practice without knowing the  solution of the optimization problem or at least a bound of $| \mathbb{E}[Z(\xi^\star)] |$.    The second algorithm which we compare to our algorithms is the UCB-V Algorithm (see \cite[Section 3.1]{audibert2009exploration}). It consists in only resampling the random variable whose confidence interval has the highest upper bound. To do so, we replace Steps 3 to 6 of Algorithm \ref{alg:bandit} by: 
\begin{align*}
&\text{Compute $\xi^+ = \arg \max_{\xi \in \Xi_n} ~ \beta^+_{\xi,m(\xi)}$},\\
&\text{Sample $Z(\xi^+)$, increment $m(\xi^+)$ and update $\epsilon_{\xi^+,m(\xi^+)}$.}
\end{align*}
We choose these algorithms to perform the comparison because i) ME Algorithm ensures theoretical guarantees similar to ours (although in absolute precision) and ii) the UCB-V Algorithm is optimal, in a sense that we will define later, for solving the optimization problem \eqref{eq:problem}. \\
 
We illustrate on Figure \ref{fig:compMEadaptivenonadaptive} the behavior of algorithms. The results that we show on Figure \ref{fig:compMEadaptivenonadaptive} are the ones of a single run of each algorithm. On the left scale, we plot the estimates $\hat{\mathbb{E}}_{m(\xi)}[Z(\xi)]$ as defined in \eqref{eq:newestimate} and the associated confidence intervals $[\beta^-_{\xi,m(\xi)},\beta^+_{\xi,m(\xi)} ]$ of level $1-d_{m(\xi)}$ given by \eqref{eq:betaconfint}. The estimates and confidence intervals for $\xi \in \Xi_n$ are drawn in blue, while the ones for $\xi \notin \Xi_n$ are drawn in red. On the right scale, we plot the number of samples $m(\xi)$ generated for each $\xi\in \Xi$. We observe that Algorithm \ref{alg:nonadaptatif} samples too much the random variables with low expectation in absolute value. This is responsible for the three peaks on $m(\xi)$ observed on Figure \ref{fig:sub-first}. Algorithm \ref{alg:bandit} avoids this drawback as it does not try to reach the condition $\epsilon_{\xi,m(\xi)} < 1$ for all random variables. The UCB-V algorithm samples mostly the two random variables with highest expectations (more than 99\% of the samples are drawn from these random variables). Other random variables are not sufficiently often sampled for reaching rapidly the stopping condition based on confidence intervals. The Median Elimination Algorithm oversamples all random variables in comparison with other algorithms. 
\begin{figure}[h]
\begin{subfigure}{.5\textwidth}
  \centering
  \includegraphics[width=.9\linewidth]{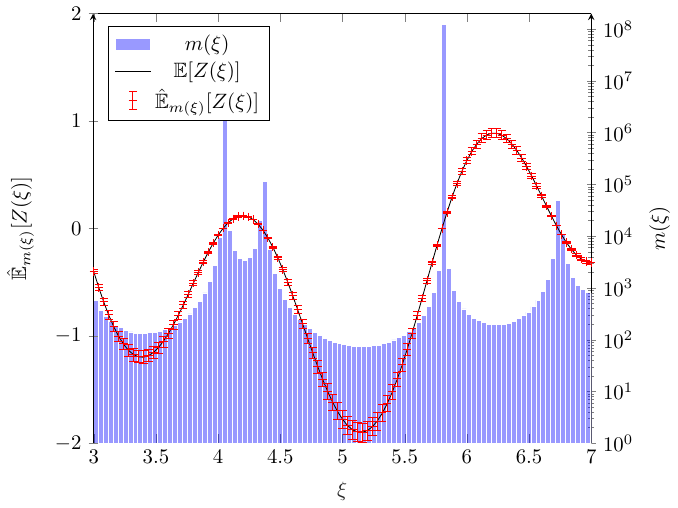}
  \caption{Algorithm \ref{alg:nonadaptatif}}
  \label{fig:sub-first}
\end{subfigure}
\begin{subfigure}{.5\textwidth}
  \centering
  \includegraphics[width=.9\linewidth]{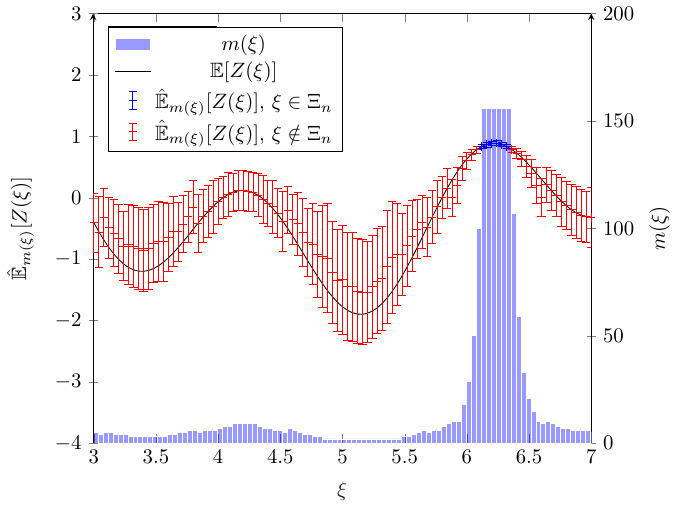}
  \caption{Algorithm \ref{alg:bandit}}
  \label{fig:sub-second}
\end{subfigure}
~
\newline
~
\begin{subfigure}{.5\textwidth}
  \centering
  \includegraphics[width=.9\linewidth]{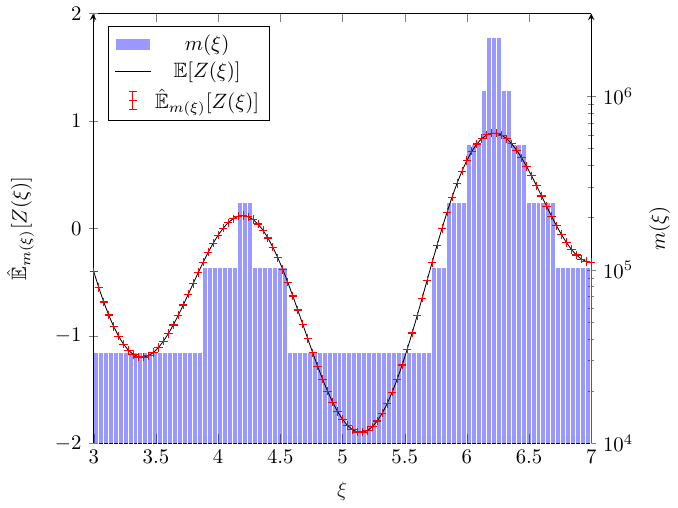}
  \caption{ME Algorithm}
  \label{fig:sub-third}
\end{subfigure}
\begin{subfigure}{.5\textwidth}
  \centering
  \includegraphics[width=.9\linewidth]{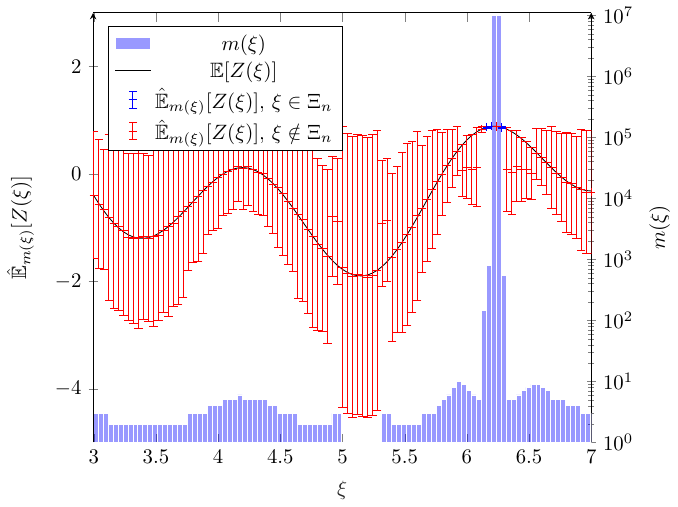}
  \caption{UCB-V Algorithm}
  \label{fig:sub-fourth}
\end{subfigure}
\caption{Final state of each algorithm after one run with $\tau = 0.1$, $\lambda = 0.1$ and $\tau_{abs} = \tau | \mathbb{E}[Z(\xi^\star)] |$ for ME Algorithm. Left scale : values of the estimates $\hat{\mathbb{E}}_{m(\xi)}[Z(\xi)]$ together with the associated confidence intervals of level $1 - d_{m(\xi)}$. Right scale : values of $m(\xi)$.}
\label{fig:compMEadaptivenonadaptive}
\end{figure}
\paragraph{Complexity.}
To perform a quantitative comparison with existing algorithms in the case of costly sampling, a relevant complexity measure is the total number of samples generated after a single run of the algorithm
$$\mathcal{M} = \sum_{\xi \in \Xi} m(\xi).$$
Table \ref{tab:complexity} shows the average complexity  $\mathbb{E}(\mathcal{M})$ estimated using $30$ independent runs of each algorithm. We observe that the expected complexity of Algorithm \ref{alg:bandit} is far below the one of the other algorithms. It means that, for the complexity measure $\mathbb{E}(\mathcal{M})$, the adaptive algorithm we have proposed  performs the best. \\

\begin{table}[H]
\centering
\begin{tabular}{|c|c|c|c|}
\hline 
 ME Alg. & Alg. \ref{alg:nonadaptatif} & Alg. \ref{alg:bandit} & UCB-V Alg. \\ 
\hline  \hline
$2.0 \cdot 10^7$  & $1.4 \cdot 10^8$ & $1.9 \cdot 10^3$ & $1.9 \cdot 10^8$ \\ 
\hline 
\end{tabular} 
\caption{Average complexity $\mathbb{E}(\mathcal{M})$, estimated using 30 runs for each algorithm, with $\tau =0.1$, $\lambda = 0.1$ and $\tau_{abs} = \tau | \mathbb{E}[Z(\xi^\star)] |$ for ME algorithm. }
\label{tab:complexity}
\end{table}

We now compare the four algorithms in terms of expected runtime, that is a measure of complexity taking into account the sampling cost and the cost of all other operations performed by the algorithms. 
Denoting by $t^\star$ the time (assumed constant) for generating one sample from a distribution, the runtime of an algorithm is a random variable  $T = \mathcal{M} t^\star + \mathcal{N}$, where $\mathcal{M} t^\star$
is the sampling time, and $\mathcal{N}$ is the (random) time taken by all other operations. The expected runtime is then 
$\mathbb{E}(T) = \mathbb{E}(\mathcal{M}) t^\star + \mathbb{E}(\mathcal{N})$. From the values of $\mathbb{E}(\mathcal{N})$ and $\mathbb{E}(\mathcal{M})$, estimated over 30 runs of the algorihms, we deduce Table \ref{tab:runtimecostly}, which shows the average runtime $\mathbb{E}(T)$ for different values of $t^\star$. We  observe that Algorithm \ref{alg:bandit} has the smallest average runtime whatever the sampling cost. The first line corresponds to $\mathbb{E}(\mathcal{N})$ and shows that Algorithm \ref{alg:bandit} performs the best when sampling cost $t^\star=0$ (or negligible). The
 impressive gain for large sampling costs $t^\star$ is due to the small value of the average number of samples $\mathbb{E}(\mathcal{M})$ required by the algorithm.
\begin{table}[H]
\centering
\begin{tabular}{|c|c|c|c|c|}
\hline 
$t^\star$ & ME Alg. & Alg. \ref{alg:nonadaptatif} & Alg. \ref{alg:bandit} & UCB-V Alg. \\ 
\hline 
$ 0$ & $2.5911$ & $5.0 \cdot 10^1$  & $3 \cdot 10^{-3}$ & $1.2 \cdot 10^3$ \\ 
\hline 
$ 10^{-6}$ & $2.0 \cdot 10^1$ & $1.9 \cdot 10^2$ & $3.8 \cdot 10^{-3}$ & $1.4 \cdot 10^{3}$ \\ 
\hline 
$ 10^{-4}$ & $2.0 \cdot 10^3$ & $1.9 \cdot 10^4$ & $3.8 \cdot 10^{-1}$ & $1.4 \cdot 10^{5}$ \\ 
\hline 
$ 10^{-2}$ & $2.0 \cdot 10^5$ & $1.9 \cdot 10^6$ & $3.8 \cdot 10^{1}$ & $1.4 \cdot 10^{7}$ \\ 
\hline 
$1$ & $2.0 \cdot 10^7$ & $1.9 \cdot 10^8$ & $3.8 \cdot 10^{3}$ & $1.4 \cdot 10^{9}$ \\ 
\hline 
\end{tabular} 
\caption{Estimated runtime $T$ (in seconds) for different values of $t^\star$, with $\tau = \lambda = 0.1$ and $\tau_{abs} = \tau | \mathbb{E}[Z(\xi^\star)] |$ for ME algorithm.}
\label{tab:runtimecostly}
\end{table}

\paragraph{Behavior of Algorithm \ref{alg:bandit}.}
Now, we illustrate  the behavior of  Algorithm \ref{alg:bandit} on Figure \ref{fig:evolutionwithn} and show the evolution with $n$ of $\Xi_n$ and $m_n(\xi)$ for a single run of Algorithm \ref{alg:bandit}, where $m_n(\xi)$ denotes the total number of samples from $Z(\xi)$ generated from iteration $1$ to iteration $n$.  When $n=1$, the algorithm has sampled every random variable once, which is not enough to distinguish some confidence intervals. So $\Xi_1$ is equal to $\Xi$. When $n=10$, some confidence intervals can be distinguished and the algorithm has identified two groups of values where a quasi-maximum could be. These two groups correspond to the two groups of random variables in $\Xi_{10}$. When $n=21$, the algorithm has identified the main peak of the function. However, the values of $\epsilon_{\xi,m(\xi)}$ for $\xi$ in $\Xi_{21}$ are not small enough for the algorithm to stop. Then the algorithm continues sampling the random variables in $\Xi_n$, updating $\Xi_n$ when it is necessary. $\epsilon_{\xi,m(\xi)}$ for $\xi$ in $\Xi_n$ decreases since $m(\xi)$ is increasing for these values of $\xi$ and the algorithm stops at $n=214$ when $\underset{\xi \in \Xi_{211}}{\max} ~ \epsilon_{\xi,m(\xi)} < \frac{\tau}{2+\tau}$. \\

Figure \ref{fig:complexitytaulambda} shows the influence of $\tau$ and $\lambda$ on the average complexity $\mathbb{E}(\mathcal{M})$ of Algorithm \ref{alg:bandit}. We observe that $\tau $ has a much bigger impact than $\lambda$. This observation is consistant with the impact of 
 $\epsilon = \tau/(2+\tau)$ and $\delta = \lambda/\#\Xi$ on the expected number of sampling $\mathbb{E}(M)$ to build an estimate $\hat{\mathbb{E}}_{M}[Z]$ of $\mathbb{E}[Z(\xi)]$ with relative precision $\epsilon$ with probability $1-\delta$ (see Remark \ref{rem:complexity}).
\begin{figure}[H]
\begin{subfigure}{.5\textwidth}
  \centering
  \includegraphics[width=.9\linewidth]{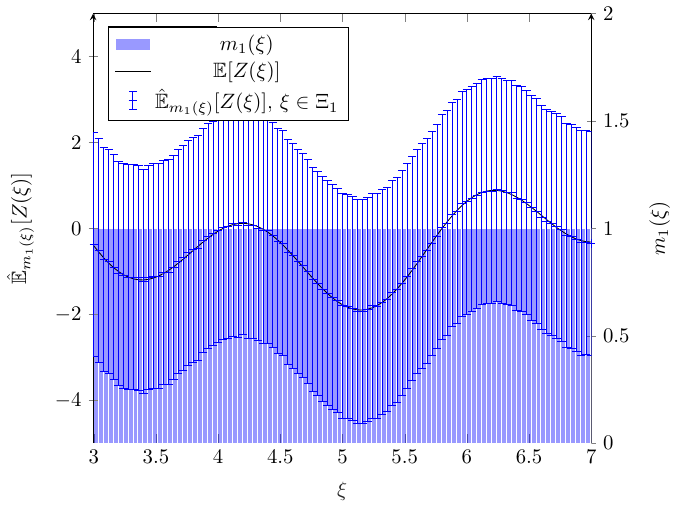}
  \caption{$n=1$}
  \label{fig1:sub-first}
\end{subfigure}
\begin{subfigure}{.5\textwidth}
  \centering
  \includegraphics[width=.9\linewidth]{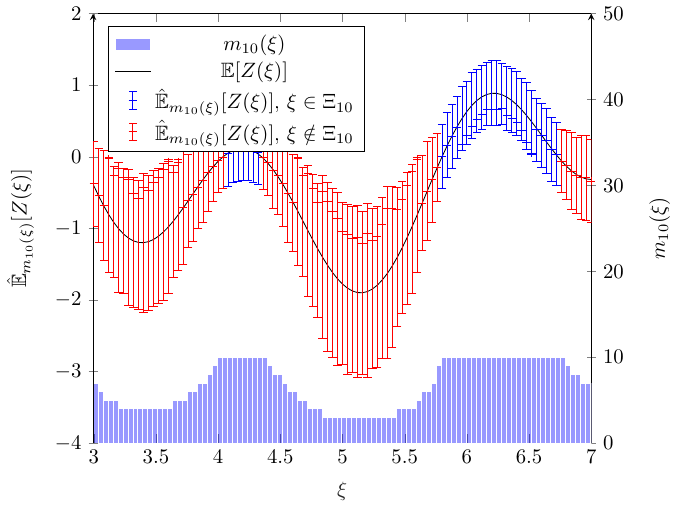}
  \caption{$n=10$}
  \label{fig1:sub-second}
\end{subfigure}
\\
\begin{subfigure}{.5\textwidth}
  \centering
  \includegraphics[width=.9\linewidth]{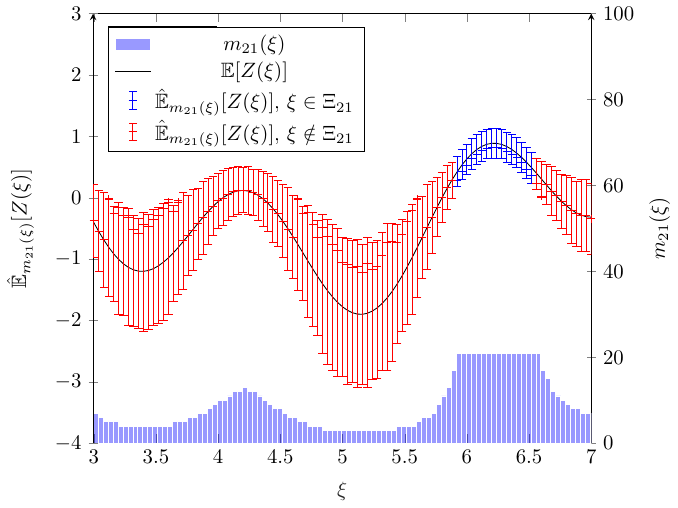}
  \caption{$n=21$}
  \label{fig1:sub-third}
\end{subfigure}
\begin{subfigure}{.5\textwidth}
  \centering
 \includegraphics[width=.9\linewidth]{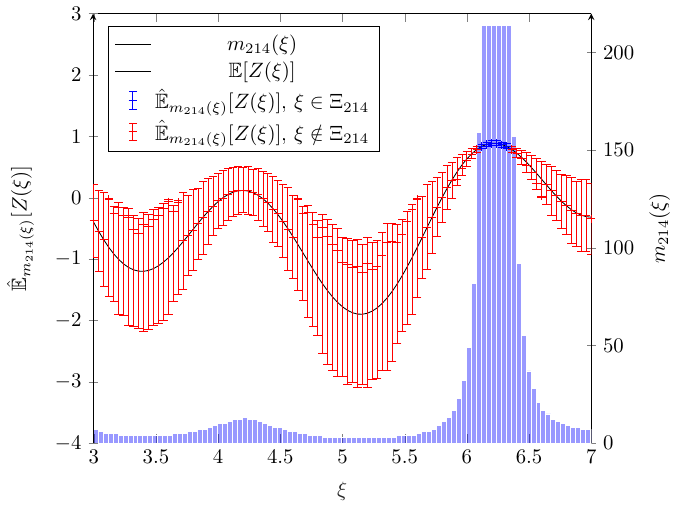}
  \caption{$n=214$}
  \label{fig1:sub-fourth}
\end{subfigure}
\caption{Evolution of $\Xi_n$ and number of samples  $m_n(\xi)$ with  $n$ for Algorithm \ref{alg:bandit} with $\tau = \lambda = 0.1$. }
\label{fig:evolutionwithn}
\end{figure}

\begin{figure}[h]
\centering
\includegraphics[scale=0.6]{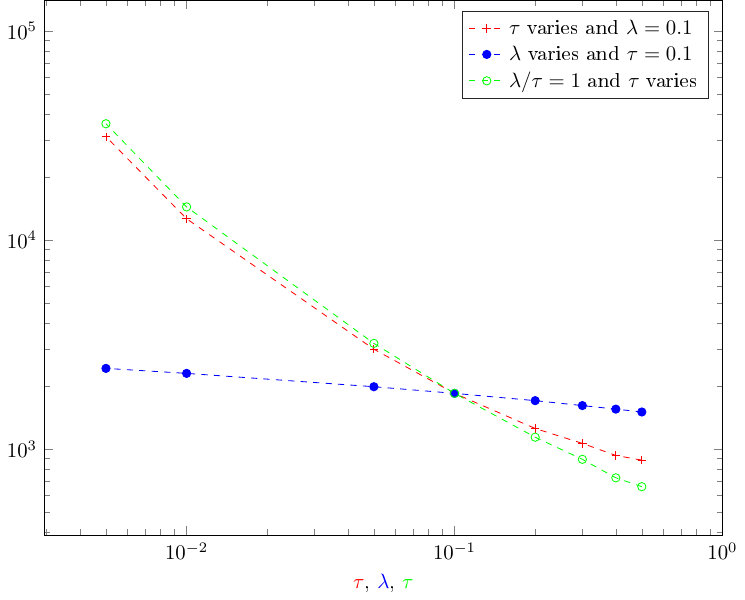}
\caption{Average complexity $\mathbb{E}(\mathcal{M})$ of Algorithm \ref{alg:bandit} with respect to $\tau$ and $\lambda$ (in log-log scale).}
\label{fig:complexitytaulambda}
\end{figure} 

\subsection{\textcolor{black}{Second test case}}
\textcolor{black}{As in \cite{kaufmann2013information}, we consider Bernoulli random variables $Z(\xi), ~\xi \in \{1, \ldots,K \}$, with $K$ varying from $10$ to $60$. Concentration inequality bounds from Theorem \ref{th:hoeffding} are used. The values of $\mathbb{E}[Z(\xi)]$ for $\xi \in \{1, \ldots,K \}$ are drawn uniformly in $[0,1]$. We first estimate the probability of failure of Algorithm \ref{alg:bandit} over 1000 draws of a $K$-tuple for the means, this for each value of $K$. It corresponds to the probability of returning $\hat{\xi}$ such that $\mathbb{E}[Z(\xi^\star)] - \mathbb{E}[Z(\hat{\xi})] > \tau | \mathbb{E}[Z(\xi^\star)] |$. The parameters used for these runs are $\tau = 0.5$, $\lambda = 0.5$ and $p=2$. On these runs, and on all the runs of Algorithm \ref{alg:bandit} we performed in this paper, the estimated probability of failure is zero, which means that Algorithm \ref{alg:bandit} did not fail. This could be explained by the use of concentration inequalities that are not very sharp and a very conservative stopping criterion. It is mostly due to the union bound argument used in our proof, see, e.g., the one from Proposition \ref{prop:banditresult}.} \\

\textcolor{black}{We observe on Figure \ref{fig:compK_tau_lambda} the complexity $\mathbb{E}(\mathcal{M})$ plotted as a function of $K$ for different values of $\tau$ and $\lambda$. Again for this case, we observe that $\tau$ has more influence than $\lambda$ on the complexity $\mathbb{E}(\mathcal{M})$. Also $\mathbb{E}(\mathcal{M})$ does not deteriorate when the number $K$ of parameters in $\Xi$ increases.}


\begin{figure}[h]
\centering
\includegraphics[scale=0.6]{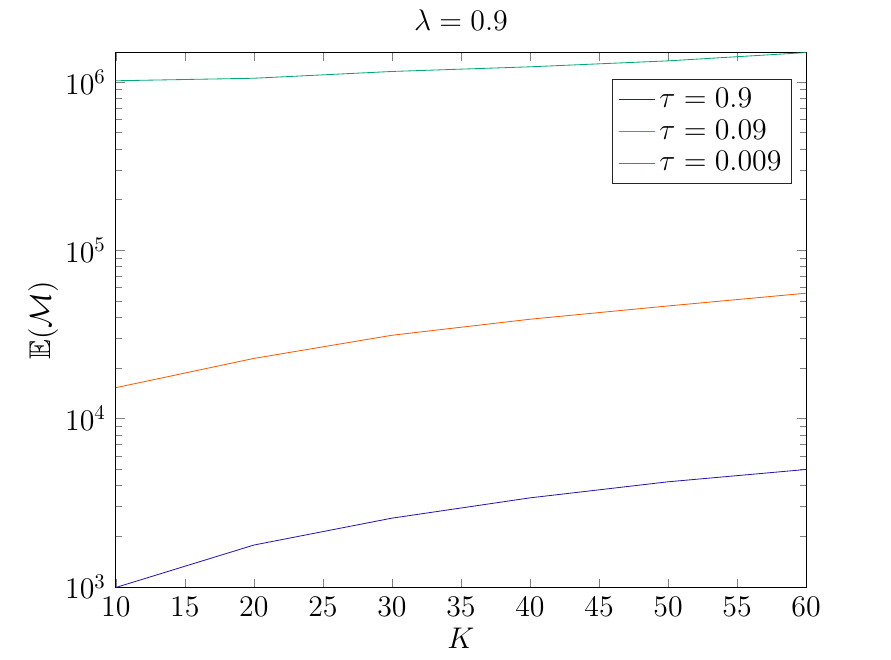}  \\
\includegraphics[scale=0.6]{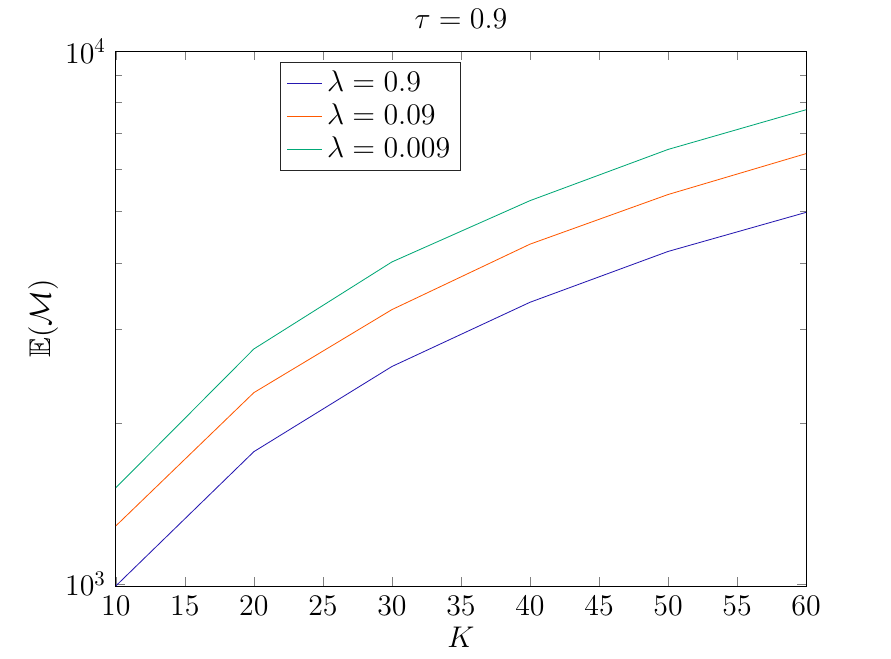}
\caption{\textcolor{black}{$\mathbb{E}(\mathcal{M})$ averaged for the 1000 draws of the means. (top) $\tau$ varies and $\lambda = 0.9$. (bottom) $\lambda$ varies and $\tau = 0.9$.}}
\label{fig:compK_tau_lambda}
\end{figure}

\subsection{\textcolor{black}{Summary of performance validation}}
\textcolor{black}{We considered in this paper bandit problems where the random variables are costly to sample. It is not a traditional framework. Our adaptive algorithm clearly outperforms the others by several orders of magnitude in terms of sampling complexity (see Table \ref{tab:complexity}). Therefore the more costly the sampling, the more efficient our adaptive strategy (see Table \ref{tab:runtimecostly}). Sampling is costly in two typical situations: the samples of $Z(\xi)$ are drawn from a high-dimensional distribution or $Z(\xi) = f_\xi(X)$ with $X$ a random variable (not necessarily costly to sample) and $f_\xi$ a function costly to evaluate. The latter situation is typical for the computer experiment community.} \\

\textcolor{black}{On both Figure \ref{fig:complexitytaulambda} and Figure \ref{fig:compK_tau_lambda}, we observed the high impact of $\tau$, in comparison of the impact of $\lambda$, on the complexity $\mathbb{E}[\mathcal{M}]$. It allows us to use the adaptive algorithm with sets $\Xi$ of high cardinal even with small values of $\lambda$.} \\

\bibliographystyle{spmpsci}     
\bibliography{biblio}

\begin{thebibliography}{10}
\providecommand{\url}[1]{{#1}}
\providecommand{\urlprefix}{URL }
\expandafter\ifx\csname urlstyle\endcsname\relax
  \providecommand{\doi}[1]{DOI~\discretionary{}{}{}#1}\else
  \providecommand{\doi}{DOI~\discretionary{}{}{}\begingroup
  \urlstyle{rm}\Url}\fi

\bibitem{Audibert2011}
Audibert, J.Y., Bubeck, S., Munos, R.: Bandit view on noisy optimization.
\newblock Optimization for Machine Learning p. 431 (2011)

\bibitem{audibert2009exploration}
Audibert, J.Y., Munos, R., Szepesv{\'a}ri, C.: Exploration--exploitation
  tradeoff using variance estimates in multi-armed bandits.
\newblock Theoretical Computer Science \textbf{410}(19), 1876--1902 (2009)

\bibitem{Beyer2007}
Beyer, H.G., Sendhoff, B.: Robust optimization--a comprehensive survey.
\newblock Computer methods in applied mechanics and engineering
  \textbf{196}(33-34), 3190--3218 (2007)

\bibitem{Bubeck2011}
Bubeck, S., Munos, R., Stoltz, G.: Pure exploration in finitely-armed and
  continuous-armed bandits.
\newblock Theoretical Computer Science \textbf{412}(19), 1832--1852 (2011)

\bibitem{Dupa1982}
Dupa, V., Herkenrath, U.: Stochastic approximation on a discrete set and the
  multi-armed.
\newblock Sequential Analysis \textbf{1}(1), 1--25 (1982)

\bibitem{Evendar2002}
Even-Dar, E., Mannor, S., Mansour, Y.: Pac bounds for multi-armed bandit and
  markov decision processes.
\newblock In: International Conference on Computational Learning Theory, pp.
  255--270. Springer (2002)

\bibitem{Garivier2011}
Garivier, A., Capp{\'e}, O.: The kl-ucb algorithm for bounded stochastic
  bandits and beyond.
\newblock In: Proceedings of the 24th annual conference on learning theory, pp.
  359--376 (2011)

\bibitem{Gong2000}
Gong, W.B., Ho, Y.C., Zhai, W.: Stochastic comparison algorithm for discrete
  optimization with estimation.
\newblock SIAM Journal on Optimization \textbf{10}(2), 384--404 (2000)

\bibitem{Audibert2010}
J.-Y.~Audibert S.~Bubeck, R.M.: Best arm identification in multi-armed bandits.
\newblock In: Annual Conference on Learning Theory (COLT) (2010)

\bibitem{Kal2012}
Kalyanakrishnan, S., Tewari, A., Auer, P., Stone, P.: Pac subset selection in
  stochastic multi-armed bandits.
\newblock In: ICML, vol.~12, pp. 655--662 (2012)

\bibitem{Kano2019}
Kano, H., Honda, J., Sakamaki, K., Matsuura, K., Nakamura, A., Sugiyama, M.:
  Good arm identification via bandit feedback.
\newblock Machine Learning \textbf{108}(5), 721--745 (2019)

\bibitem{Kaufmann2016}
Kaufmann, E., Capp{\'e}, O., Garivier, A.: On the complexity of best-arm
  identification in multi-armed bandit models.
\newblock The Journal of Machine Learning Research \textbf{17}(1), 1--42 (2016)

\bibitem{kaufmann2013information}
Kaufmann, E., Kalyanakrishnan, S.: Information complexity in bandit subset
  selection.
\newblock In: Conference on Learning Theory, pp. 228--251. PMLR (2013)

\bibitem{Lattimore2020}
Lattimore, T., Szepesv{\'a}ri, C.: Bandit algorithms.
\newblock Cambridge University Press (2020)

\bibitem{Locatelli2016}
Locatelli, A., Gutzeit, M., Carpentier, A.: An optimal algorithm for the
  thresholding bandit problem.
\newblock arXiv preprint arXiv:1605.08671  (2016)

\bibitem{Mnihthesis}
Mnih, V.: Efficient stopping rules.
\newblock Ph.D. thesis, University of Alberta (2008)

\bibitem{Mnih2008}
Mnih, V., Szepesv{\'a}ri, C., Audibert, J.Y.: Empirical bernstein stopping.
\newblock In: Proceedings of the 25th international conference on Machine
  learning, pp. 672--679 (2008)

\bibitem{Muk2017}
Mukherjee, S., Naveen, K.P., Sudarsanam, N., Ravindran, B.: Thresholding
  bandits with augmented ucb.
\newblock arXiv preprint arXiv:1704.02281  (2017)

\bibitem{Nemirovski2009}
Nemirovski, A., Juditsky, A., Lan, G., Shapiro, A.: Robust stochastic
  approximation approach to stochastic programming.
\newblock SIAM Journal on optimization \textbf{19}(4), 1574--1609 (2009)

\bibitem{Sutton2018}
Sutton, R.S., Barto, A.G.: Reinforcement learning: An introduction.
\newblock MIT press (2018)

\bibitem{Tao2019}
Tao, C., Blanco, S., Peng, J., Zhou, Y.: Thresholding bandit with optimal
  aggregate regret.
\newblock In: Advances in Neural Information Processing Systems, pp.
  11664--11673 (2019)

\bibitem{kuleshov2014algorithms}
\textcolor{black}{Kuleshov, V. and Precup, D.}: \textcolor{black}{Algorithms for
  multi-armed bandit problems}.
\newblock \textcolor{black}{arXiv preprint arXiv:1402.6028}
  (\textcolor{black}{2014})

\bibitem{Yan1992}
Yan, D., Mukai, H.: Stochastic discrete optimization.
\newblock SIAM Journal on control and optimization \textbf{30}(3), 594--612
  (1992)

\end{thebibliography}

\appendix

\section{Intermediate results} 
\label{sec:intermerdiate}
Here we provide  intermediate results used thereafter for the proof of Proposition \ref{prop:complexityMbernstein} in Section \ref{sec:prop3.4}. 
We first recall a version of Bennett's inequality from \cite[Lemma 5]{audibert2009exploration}.
\begin{lemma}
Let $U$ be a random variable defined on $(\Omega, \mathcal{F}, \mathbb{P})$ such that $U \le b$  almost surely, with $b \in \mathbb{R}$. Let $U_1, \ldots, U_m$ be i.i.d. copies of $U$ and $\overline{U}_{\ell} = \frac{1}{\ell}\sum_{i=1}^{\ell} U_i$. For any $x>0$, it holds, with probability at least $1 - \exp(-x)$, 
simultaneously for all $1 \le \ell \le m$
\begin{equation}
\ell \left( \overline{U}_\ell - \mathbb{E} \left[ U \right] \right) \le \sqrt{2m \mathbb{E} \left[ U^2 \right] x} +  b_+ x / 3, 
\label{eq:lemma1}
\end{equation}
with $b_+ = \max(0, b)$.
\label{lemma:base}
\end{lemma}
Now, the following result provides a bound with high probability for the estimated variance of an i.i.d. sequence of bounded random variables.
\begin{lemma}
Let $X$ be a  bounded random variable defined on $(\Omega, \mathcal{F}, \mathbb{P})$, such that $a \le X  \le b$  almost surely, with $a<b$ two real numbers. Let $X_1, \ldots, X_m$ be i.i.d. copies of $X$ and $\overline{V}_m = \dfrac{1}{m} \sum_{i=1}^m (X_i - \overline{X}_m)^2$ where $\overline{X}_m = \frac{1}{m}\sum_{i=1}^m X_i$. Then, for any $x>0$
\begin{equation}
\mathbb{P} \left( \overline{V}_m \le \mathbb{V}[X] + \sqrt{2\mathbb{V}[X] \dfrac{(b-a)^2x}{m}} + \dfrac{x(b-a)^2}{3m} \right) \ge 1 - \exp(-x). 
\end{equation}  
\label{lemma:varianceempirique}
\end{lemma}
\begin{proof} Let us define $U= (X - \mathbb{E}[X])^2$ which satisfies $U \le (b-a)^2$ almost surely. Applying Lemma \ref{lemma:base} with $U$ defined previously with $\ell=m$ gives for any $x>0$ 
$$
\mathbb{P} \left( m \left( \overline{U}_m - \mathbb{E}[U] \right) \le  \sqrt{2m \mathbb{E}[U^2]x} +  \dfrac{x(b-a)^2}{3}\right) \ge 1 - \exp(-x).
$$
Moreover, as $\overline{U}_m  = \overline{V}_m + (\overline{X}_m - \mathbb{E}[X])^2$ and using the boundedness of $U$ we get
\begin{equation*}
\mathbb{P} \left( \overline{V}_m  \le \mathbb{E}[U] +  \sqrt{2\mathbb{E}[U] \dfrac{(b-a)^2x}{m}} + \dfrac{x(b-a)^2}{3m} \right)  \ge 1 - \exp(-x),
\end{equation*} 
which ends the proof since $\mathbb{E}[U] = \mathbb{V}[X]$. \qed
\end{proof}
We recall a second result in the line of \cite[Lemma 3]{Mnihthesis}.
\begin{lemma}
Let $q,k$ be positive real numbers. If $t >0$ is a solution of 
\begin{equation}
 \frac{\log qt}{t} = k, \label{eq:log1}
\end{equation}
then 
\begin{equation}
t \le  \dfrac{2}{k} \log \dfrac{2q}{k}.
\label{eq:log2}
\end{equation}
\item 
Moreover, if $t'$ is such that 
\begin{equation}
t^{'} \ge \dfrac{2}{k} \log \dfrac{2q}{k},
\label{eq:log3}
\end{equation}
then
 \begin{equation}
\dfrac{\log qt^{'}}{t^{'}} \le k.
\label{eq:log4}
\end{equation}
\label{lemma:log}
\end{lemma}
\begin{proof} Let  $t>0$ be a solution  of \eqref{eq:log1}. Since the function $\log$ is concave, it holds for all $s>0$
$$
 kt = \log(q t)\le   \log(q s) + \frac{ t - s}{s} .
$$
In particular, for  $s =  \frac{2}{k} > 0$  we get 
\begin{equation}
t \le \frac{2}{k} \left( \log \dfrac{2q}{k}- 1 \right) \le \frac{2}{k}  \log \dfrac{2q}{k}  ,
\label{eq:lemmelogavectermenegatif}
\end{equation}
which yields \eqref{eq:log2}. \\
Now, let $\varphi : s \mapsto \frac{\log(qs)}{s}$ defined for $s>0$. This function  is continuous, strictly increasing on $(0,  \frac e q]$ and strictly decreasing on $[\frac e q, \infty)$ so it admits a maximum at $t=\frac e q$.
The existence of a solution $t>0$ of \eqref{eq:log1} implies  $ k \le \frac q e$. 
If $ k = \frac q e$ then $t =   \frac e q$ and $\varphi(t)$  is the maximum of $\varphi$. For any $t' >0$, in particular satisfying \eqref{eq:log3}, we have $\varphi(t') \le \varphi(t) = k$ which is  \eqref{eq:log4}.  If $ 0  < k < \frac q e$, there are two solutions $t_1,t_2$ to \eqref{eq:log1}  such that $0 < t_1< \frac e q  < t_2$. By \eqref{eq:log2} and \eqref{eq:log3} we have $t'\ge t_2 >\frac e q $ and since $\varphi$ is stricly discreasing on $[\frac e q, \infty)$ it holds $\varphi(t')\le \varphi(t_2) = k$ , that is \eqref{eq:log4}.\qed
\end{proof}

\section{Proof of Proposition \ref{prop:complexityMbernstein}} \label{sec:prop3.4}

  Let us define the two events $A  = \bigcap_{m \ge 1} A_m$ and $B  =  \bigcap_{m \ge 1} B_m$ with 
$$
A_m = \left\{ \overline{V}_m \le \sigma^2 + \sqrt{2 \sigma^2 (b-a)^2 \log(3/d_m)/m} + \log(3/d_m) (b-a)^2 /3m \right \},
$$
and 
$$
B_m = \left\{ | \overline{Z}_m - \mu | \le c_m \right\}.
$$
Applying Lemma \ref{lemma:varianceempirique} with $x=\log(3/d_m)$ for  $A_m, m \ge 1$ together with a union bound argument leads to $\mathbb{P}(A) \ge 1 - \delta/3$. Similarly, using a union bound argument and  Theorem \ref{th:bernstein} with $x=\log(3/d_m)$, for $B_m, m \ge 1$, gives $\mathbb{P}(B) \ge 1 - \delta$. By gathering these two results we have 
\begin{equation}
\mathbb{P} \left( A \cap B \right) \ge 1 - \left( \mathbb{P} (A^c) +  \mathbb{P}(B^c) \right) \ge 1 - \frac{4 \delta}{3},
\label{eq:probAB}
\end{equation}
where $A^c$ and $B^c$ correspond respectively to the complementary events of $A$ and $B$. \\
It remains to  prove that $A\cap B$ implies
\begin{equation}
 M \leq \left \lceil  \frac{2}{\nu} \left[ \log\left( \frac{3}{\delta c} \right) + p\log( \frac{2p}{\nu} ) \right] \right \rceil\label{eq:majorationM}
,\end{equation}
which will {prove} \eqref{eq:Mmajoreproba}.
 In what follows, we suppose that $A \cap B$ holds.  \\
First we derive an upper bound for $\overline{V}_m$. Since $A$ holds, we have 
\begin{equation}
\label{eq:varineq}
 \overline{V}_m \le \sigma^2 + \sqrt{2 \sigma^2 (b-a)^2 \log(3/d_m)/m} + \log(3/d_m) (b-a)^2 /3m.
\end{equation}
Lemma \ref{lemma:log}  with $k = \frac{\sigma^2}{p (b-a)^2}$ and $q = \left({\frac{3}{\delta c}}\right)^{1/p}$ gives  
for any integer  $m \ge M_{\sigma^2}$
\begin{equation}
\label{eq:sigma2}
 \dfrac {(b-a)^2}{ m } \log \dfrac{ 3 }{ d_m}  \le \sigma^2, 
\end{equation}
where 
\begin{equation*}
M_{\sigma^2} = \dfrac{2(b-a)^2}{\sigma^2} \left( p \log \left( \dfrac{2 p (b-a)^2}{\sigma^2} \right) + \log \left( \dfrac{3}{c \delta} \right) \right).
\end{equation*} 
Again, Lemma \ref{lemma:log}  with $k = \frac{\epsilon^2 \mu^2}{p (b-a)^2}$ and $q = \left({\frac{3}{\delta c}}\right)^{1/p}$ gives  for any integer  $m \ge M_{\epsilon^2 \mu^2}$
\begin{equation}
\label{eq:epsmu2}
\dfrac {(b-a)^2}{ m } \log \dfrac{ 3 }{ d_m}   \le \epsilon^2 \mu^2,
\end{equation}
where 
\begin{equation*}
M_{\epsilon^2 \mu^2} = \dfrac{2(b-a)^2}{\epsilon^2 \mu^2} \left( p \log \left( \dfrac{2p (b - a)^2}{\epsilon^2 \mu^2} \right) + \log \left( \dfrac{3}{c \delta} \right)  \right).
\end{equation*} 
For all $m \ge  \min \left( M_{\sigma^2}, M_{\epsilon^2 \mu^2} \right)$, i.e. $m\ge  M_{\sigma^2}$ or $m\ge M_{\epsilon^2 \mu^2}$, we obtain from \eqref{eq:varineq} and \eqref{eq:sigma2}, or \eqref{eq:varineq} and \eqref{eq:epsmu2},  that 
\begin{equation}
\overline{V}_m \le (1 + \sqrt{2} + 1/3) \max(\sigma^2,\epsilon^2 \mu^2).
\label{eq:boundvar}
\end{equation}
In what follows, we define $\underline{M}=  \min \left( M_{\sigma^2}, M_{\epsilon^2 \mu^2} \right)$. 
Now, we deduce  from \eqref{eq:boundvar} an upper bound for  $c_m$. By definition, 
$$
c_m  =  \sqrt{\dfrac{2\overline{V}_m \log(3/d_m)}{m}} + \sqrt{\dfrac{3 (b-a)^2 \log(3/d_m)^2}{m^2} },
$$
then for all integer $m \ge \underline{M}$ and using either \eqref{eq:sigma2}, or \eqref{eq:epsmu2}, we have
\begin{equation}
c_m 
 \le  \sqrt{\dfrac{  \alpha \log (3/d_m) }{m}},
\label{eq:boundcm}
\end{equation}
with $ \alpha := (\sqrt{2 + 2\sqrt{2} + 2/3} + 3)^2 \max(\sigma^2,\epsilon^2 \mu^2)$.\\

Now, using  \eqref{eq:boundcm}, we seek a bound for $M$, the smallest integer such that $c_M \le \epsilon | \overline{Z}_M |$.
To that aim, let us introduce the integer $M^\star$, 
\begin{equation}
M^\star = \min \left \{ m \in \mathbb{N}^*: m \ge \underline{M},  \sqrt{\dfrac{\alpha \log (3/d_{m}) }{{m}}} \le \dfrac{\epsilon | \mu |}{1 + \epsilon}\right\},
\label{eq:conditionmstar}
\end{equation}
 and the integer valued random variable $M^+$
\begin{equation}
M_+ = \min \left \{ m \in \mathbb{N}^* :   c_m \le  \dfrac{\epsilon  | \mu |}{1 + \epsilon} \right \}.
\label{eq:conditionm}
\end{equation}
If $\underline{M}\ge  M_+$ then $M^\star \ge M_+$.  \\
Otherwise, $ \underline{M} < M_+$ and we have 
$
M_+ =  \min \left \{ m \ge \underline{M}  :   c_m \le  \dfrac{\epsilon  | \mu |}{1 + \epsilon} \right \}.
$
Moreover, as \eqref{eq:boundcm} holds for all  $m \ge \underline{M} $, we get the inclusion
$$
 \left \{ m \in \mathbb{N}^*: m \ge \underline{M},  \sqrt{\dfrac{\alpha \log (3/d_{m}) }{{m}}} \le \dfrac{\epsilon | \mu |}{1 + \epsilon}\right\} \subset \left\{ m \in \mathbb{N}^* : m \ge \underline{M} ,   c_m \le  \dfrac{\epsilon  | \mu |}{1 + \epsilon} \right \}. 
$$
Taking the $\min$ leads again to $M^\star \ge M_+$. Moreover, since $B$ holds, $|\mu| - c_{M_+} \le | \overline{Z}_{M_+} |$ and using \eqref{eq:conditionm} it implies that $c_{M_+} \le \epsilon | \overline{Z}_{M_+} |$. By definition of $M$  we get $M_+ \ge M$. Hence, we have $M^\star \ge M$. To conclude the proof, it remains to  find an upper bound for $M^\star$. 
Applying again Lemma \ref{lemma:log} with $k = \frac{\epsilon^2 \mu^2}{(1+\epsilon)^2 \alpha p }$ and $q = \left(\frac{3}{\delta c}\right)^{1/p}$ gives for any integer $m \ge M_f$
\begin{equation}
\label{eq:boundMf}
\dfrac{\alpha \log (3/d_{m}) }{{m}} \le \dfrac{\epsilon^2  \mu ^2}{(1 + \epsilon)^2}
\end{equation}
with 
$$
M_f = \dfrac{2(1+\epsilon)^2\alpha}{\epsilon^2 \mu^2} \left( p \log \left( \dfrac{2p(1+\epsilon)^2\alpha}{\epsilon^2 \mu^2} \right) + \log \left( \dfrac{3}{c \delta} \right)  \right).
$$
If $M_f \le \underline{M}$, \eqref{eq:conditionmstar} and \eqref{eq:boundMf} imply $M^\star =  \lceil  \underline{M}\rceil$, where $ \lceil \cdot  \rceil$ denotes the ceil function. Otherwise $M_f > \underline{M}$ and we obtain $M^\star \le \lceil M_f \rceil$. Thus, it provides the following upper bound 
$$
M^\star \le \max \left ( \lceil \underline{M}\rceil,  \lceil M_f \rceil \right ) 
=   \lceil \max \left( \underline{M} ,M_f \right ) \rceil.
$$
Introducing $\nu = \min \left(   \frac{\max(\sigma^2,\epsilon^2 \mu^2)}{(b-a)^2}  , \frac{\epsilon^2 \mu^2}{(1+\epsilon)^2 \alpha} \right)$ we have from the definition of $M_{\sigma^2}, M_{\epsilon^2\mu^2}$ and $M_f$
\begin{equation}
M^\star \le  \left \lceil \dfrac{2}{\nu} \left( p \log \left( \dfrac{2p}{\nu} \right) + \log \left( \dfrac{3}{c \delta} \right)  \right)  \right \rceil.
\label{eq:lastmajoration}
\end{equation}
Since $M^\star \ge M$ and $A \cap B$ implies \eqref{eq:lastmajoration}, we deduce that $A \cap B$ implies \eqref{eq:majorationM}, which concludes the proof of the first result. \\
Let us now prove the result in expectation. Let $K := \left \lceil \dfrac{2}{\nu} \left( p \log \left( \dfrac{2p}{\nu} \right) + \log \left( \dfrac{3}{c \delta} \right)  \right)  \right \rceil.$ We first note that 
$$
\mathbb{E}(M) = \sum_{k=0}^\infty \mathbb{P}(M>k) \le K + \sum_{k=K}^\infty \mathbb{P}(M>k) .
$$
If $M>k $, then $c_k >\epsilon \vert \bar Z_k \vert$. 
For $k\ge K$, we would like to prove that $c_k >\epsilon \vert \bar Z_k \vert$ implies $(A_k\cap B_k)^c$, or equivalently that $A_k\cap B_k$ implies 
$c_k \le \epsilon \vert \bar Z_k \vert$. For $k\ge K$, $A_k$ implies  \eqref{eq:boundcm} and \eqref{eq:boundMf}, and therefore 
$c_k \le \frac{\epsilon\vert \mu\vert}{1+\epsilon}$. Also, $B_k$ implies $\vert \mu \vert \le \vert \bar Z_k \vert + c_k$. Combining the previous inequalities, we easily conclude that $A_k\cap B_k$ implies $c_k \le \epsilon \vert \bar Z_k \vert$.
For $k\ge K$, we then have 
$\mathbb{P}(M>k) \le \mathbb{P}(c_k >\epsilon \vert \bar Z_k \vert) \le  \mathbb{P}((A_k \cap B_k)^c) \le \mathbb{P}(A_k^c) +  \mathbb{P}(B_k^c) \le 4d_k/3$, and then
$$
\mathbb{E}(M) \le K + \sum_{k=K}^\infty 4d_k/3 \le K + 4\delta/3
,$$
which ends the proof.

\end{document}